\documentclass[12pt]{amsart}
\usepackage{amssymb}
\usepackage{amsmath}
\usepackage{amsrefs}
\usepackage{hyperref}
\usepackage{bbm}
\usepackage{xcolor}

\definecolor{refkey}{gray}{.5}   
\definecolor{labelkey}{gray}{.5} 
\usepackage{graphicx}
\usepackage{tikz-cd}
\usepackage{epsfig}
\usepackage{mathtools,amsfonts}

\usepackage{comment}
\theoremstyle{plain}

\newtheorem{theorem}{Theorem}[section]
\newtheorem{lemma}[theorem]{Lemma}
\newtheorem{sublemma}[theorem]{Sublemma}

\newtheorem{question}[theorem]{Question}
\newtheorem{proposition}[theorem]{Proposition}
\newtheorem{corollary}{Corollary}
\theoremstyle{definition}

\newtheorem{remark}{Remark}

\newtheorem*{question*}{Question} 

\numberwithin{equation}{section}

\renewcommand{\phi}{\varphi}

\newcommand{\essinf}{\mathrm{essinf}}
\usepackage{amsmath}

\DeclareMathOperator{\cps}{Cap}

\DeclareMathOperator{\diam}{diam}

\allowdisplaybreaks

\title[On genericity of random coverings]{On genericity of non-uniform Dvoretzky coverings of the circle}
\thanks{
The first author is partially supported by Japan Society for the Promotion of Science (JSPS) KAKENHI Grant Number 19K03558. 
The second author is supported by the Knut and Alice Wallenberg foundation of Sweden (KAW)
}

\date{\today}

\begin{document}

\author{Michihiro Hirayama}
\author{Davit Karagulyan}

\maketitle

\begin{abstract}

The classical Dvoretzky covering problem asks for conditions on the sequence of lengths $\{\ell_n\}_{n\in \mathbb{N}}$ so that the random intervals
$I_n : = (\omega_n -(\ell_n/2), \omega_n +(\ell_n/2))$ where $\omega_n$ is a sequence of i.i.d. uniformly distributed random variables, covers any point on the circle $\mathbb{T}$ infinitely often.
We consider the Dvoretzky covering problem when the distribution of $\omega_n$ is absolutely continuous with a density function $f$. 
When $m_f=\essinf_\mathbb{T}f>0$ and the set $K_f$ of its essential infimum points satisfies $\overline{\dim}_\mathrm{B} K_f<1$, where $\overline{\dim}_\mathrm{B}$ is the upper box-counting dimension, we show that the following condition is necessary and sufficient for $\mathbb{T}$ to be $\mu_f$-Dvoretzky covered
\[
\limsup_{n \rightarrow \infty} \left(\frac{\ell_1 + \dots + \ell_n}{\ln n}\right)\geq \frac{1}{m_f}.
\]
We next show that as long as $\{\ell_n\}_{n\in \mathbb{N}}$ and $f$ satisfy the above condition and $|K_f|=0$, then a Menshov type genericity  result holds, i.e.  Dvoretzky covering can be achieved by changing $f$ on a set of arbitrarily small Lebesgue measure. 

\end{abstract}

\section{Introduction}

Let $\{ \ell_{n}\}_{n\in \mathbb{N}}$ be a sequence of positive numbers with $0<\ell_{n}<1$ and $\left\{\omega_{n}\right\}_{n \in \mathbb{N}}$ a sequence of i.i.d. random variables on the circle $\mathbb{T}=\mathbb{R}/\mathbb{Z}$. 
Consider the random intervals
\[
I_{n}:=\left(\omega_{n}-r_{n}, \omega_{n}+r_{n}\right),
\]
where $r_n=\ell_{n} / 2$. 
If $\left\{\omega_{n}\right\}_{n\in \mathbb{N}}$ are uniformly distributed on $\mathbb{T}$, the Borel-Cantelli lemma assures that Lebesgue almost every (a.e. for short) point on $\mathbb{T}$ is covered infinitely often with probability one if and only if $\sum_{n=1}^{\infty} \ell_{n}=\infty$. 
Moreover, $\sum_{n=1}^{\infty} \ell_{n}<\infty$ implies that Lebesgue a.e. point in $\mathbb{T}$ is covered finitely often with probability one. 
In \cite{Dvoretzky1956}, Dvoretzky  asked the question whether $\mathbb{T}= \varlimsup_{n \rightarrow \infty} I_{n}(\omega)$ a.s. This condition means that every point on $\mathbb{T}$ is covered infinitely often with the intervals $I_n$.
Kahane \cite{Kahane1959} proved that $\mathbb{T}=\varlimsup_{n\to \infty } I_{n}(\omega)$ a.s. when $\ell_{n}=c/n$ with $c>1$, but this is not the case with $0<c<1$ as was proved by Billard \cite{Billard1965}. 
The case $\ell_{n}=1/n$ was solved by Mandelbrot \cite{Mandelbrot1972}. 
A complete answer was obtained by Shepp \cite{Shepp1972}: $\mathbb{T}=\varlimsup_{n\to \infty} I_{n}(\omega)$ a.s. if and only if
\[
\sum_{n=1}^{\infty} \frac{1}{n^{2}} \exp \left(\ell_{1}+\cdots+\ell_{n}\right)=\infty .
\]

We remark that in \cite{Fan-Karagulyan2021} the authors extend this result to the case of all absolutely continuous measures: more specifically they show that for the absolutely continuous measure with the density $f$ there is Dvoretzky covering for the sequence $\ell_{n}=\frac{c}{n}$ if and only if $m_f c\geq 1$, where $m_f$ is the essential infimum of $f$.

The Dvoretzky covering problem has also been studied for compact sets $F$. 
It is know that a compact set $F \subset \mathbb{T}$ is covered if and only if 
\begin{equation}\label{cond:Kahane}
	{\rm Cap}(F)=0
\end{equation}
where \eqref{cond:Kahane} means that for any Borel probability measure $\sigma$ supported on $F$ we have
\[
\int_{\mathbb{T}} \int_{\mathbb{T}} \Phi(t,s) d\sigma(t) d\sigma(s) =+\infty,
\]
where
\[
 \Phi(t, s) = \exp \left\{ \sum_{n=1}^\infty (\ell_n -|t-s|)_+ \right\}.
\]
This is called {\em Kahane's condition}, which was obtained by J-P. Kahane \cite{Kahane1990}. 
If \eqref{cond:Kahane} does not hold we write ${\rm Cap}(F)>0$.

For further information on the Dvoretzky random covering, we refer the reader to \cite{Kahane1985}. 
The Hausdorff dimension of the set that is covered infinitely often by the sequence $l_{n}=1 / n^{\alpha}$ was first calculated by Fan and Wu \cite{Fan-Wu2004}. The Hausdorff dimensions of the set of the form $I(x)=\limsup _{n\to \infty}J_n(x)$, where $J_n(x)=\left(x_n-r_n, x_n+r_{n}\right)$ and $x_{n+1}= 2 x_n \pmod 1$, was calculated by Fan, Schmeling and Troubetzkoy in \cite{Fan-Schmeling-Troubetzkoy2013}. 
Persson and Rams \cite{Persson-Rams2017} considered the same problem for general piecewise expanding maps.

\subsection{Non-uniform Dvoretzky covering}

The Dvoretzky covering problem for non-uniform measures, i.e. when $\omega_n$ are no longer uniformly distributed, is very subtle and there is little known. Even the case of absolute continuous measures is not well understood. For the uniform density one can use methods from \cite{Jasson1986} or Poisson covering techniques introduced in \cite{Shepp1972}. However, these methods are based on the homogeneity of the uniform density and are thus less practical for non-uniform densities. In this paper we provide a condition such that if it is satisfied by $\omega_n$ and $\{\ell_n\}_{n \geq 1}$, then Dvoretzky covering can be achieved by changing the density of $\omega_n$ on an arbitrary small Lebesgue measure (see Theorem \ref{mainThm:perturbation}). The result holds for a large class of densities. In harmonic analysis, this type of genericity results are known as Menshov type.

In \cite{Tang2012}, J. Tang studies Dvoretzky covering for singular measure and finds the optimal covering exponent $t_0$ for the sequence $1/n^t$, i.e. so that there is covering for the sequence $1/n^t$ when $t<t_0$, and there is no, if $t>t_0$.
We remark that when $\mu$ is absolutely continuous then $t_0=1$. 

When $\mu$ allows absolute continuous density and the density function $f$ is `flat" (for example Lipschitz) around the set $K_f$ of its essential infimum points, in \cite{Fan-Karagulyan2021} the authors give a necessary and sufficient condition for covering. In the case of the sequence $\ell_n=c/n$ the authors give a necessary and sufficient condition without imposing any regularity assumptions.  

    In this paper, we obtain results for absolute continuous measures that do not impose any assumptions on the regularity of the density $f$. Our argument uses some ideas from the study of convergence properties of convolution operators in tangential directions. This part is of independent interest.

To distinguish the classical and non-classical cases we will denote the probability measure admitting a (non-uniform) density $f\in L^1(\mathbb{T})$ by $\mu_f$ and will say that $\mathbb{T}$ is $\mu_f$-Dvoretzky covered if $\mathbb{T}= \varlimsup_{n \rightarrow \infty} I_{n}(\omega)$ a.s..

\subsubsection{Statement of results}

We now state out first Theorem: 

\begin{theorem} \label{mainThm:ub_dvoretzky-hawkes}
Let $f\in L^1(\mathbb{T})$ be a probability density function and $\{\ell_n\}_{n \in \mathbb{N}}$ be a decreasing sequence of positive numbers with $0<\ell_{n}<1$.  
\begin{enumerate}
    \item Then the following condition is necessary for $\mu_f$-Dvoretzky covering:
    \begin{equation} \label{mD}
    \limsup_{n \rightarrow \infty}\left(\frac{\ell_1+ \dots + \ell_n}{\ln n}\right)\geq \frac{1}{m_f},
    \end{equation} 
    with the convention that $1/m_f=\infty$ when $m_f=0$. If $m_f>0$ and the inequality above is strict, then \eqref{mD} is also sufficient for $\mu_f$-Dvoretzky covering.
    \item If $m_f>0$ and $\overline{\dim} _\mathrm{B} K_f<1$, then \eqref{mD} is  necessary and sufficient for covering.  
\end{enumerate}
\end{theorem}

Here $\overline{\dim} _\mathrm{B}$ denotes the upper box-counting dimension. 

\begin{remark}
Theorem \ref{mainThm:ub_dvoretzky-hawkes} is also true under the weaker assumption $\dim_\mathrm{P} K_f<1$, where $\dim_\mathrm{P}$ is the packing dimension. 
(We refer the reader to \cite{Falconer1990}*{Section 3.4} for its definition and properties.) 
By \cite{Falconer1990}*{Proposition 3.8} we have $\dim_\mathrm{P}F=\inf \{ \sup _{i\in \mathbb{N}}\overline{\dim} _\mathrm{B} F_i\colon F\subset \cup F_i\} $, where the infimum is taken over all countable covers $\{ F_i\}$ of $F$. 
Therefore, once Theorem \ref{mainThm:ub_dvoretzky-hawkes} is obtained one can replace $\overline{\dim} _\mathrm{B} K_f<1$ with $\dim_\mathrm{P} K_f<1$. We will omit the details of the proof. 
\end{remark}

For $s \in \mathbb{T}$ and $r>0$, let $B(s, r)$ denote the ball of radius $r$ centered at $s$.
The proof of Theorem \ref{mainThm:ub_dvoretzky-hawkes} is based on the analysis of convergence properties of the following operator 
\begin{equation}\label{estm}
f\mapsto f \ast \psi_r(s)=\frac{\sum_{k=1}^\infty \mu_f(B(t, r_k)\cap B(s,r_k))}{\sum_{k=1}^\infty (\ell_k-|t-s|)_+} 
\end{equation}
where $r=|t-s|$, and $\psi_r$ is an absolute continuous kernel which is an approximation of unity (to be defined in Section \ref{discussion}). When $\psi_r$ is an approximation of unity, then studying the almost sure convergence properties of convolution operators is a classical question in harmonic analysis. For kernels $\psi_r$, which are symmetric around the origin, one expects $f \ast \psi_r(s) \rightarrow f(s)$, for almost every $s\in \mathbb{T}$ as $r \rightarrow 0$. However, the study of convergence properties in so called "tangential" directions is much more complicated and can lead to divergent phenomena. One of the most classical examples of kernels where divergence occurs is the Poisson kernel \cite{Littlewood1927}. As it will be seen in Section \ref{mainProp:DtoS} the convergence of the operator in \eqref{estm} can be of tangential nature.

It seems natural to ask whether Theorem \ref{mainThm:ub_dvoretzky-hawkes}-(b) is true under the assumption $\overline{\dim}_\mathrm{B} K_f =1$ or a (much) weaker assumption $\dim_\mathrm{H} K_f<1$, where $\dim_\mathrm{H}$ denotes the Hausdorff dimension (recall that $\dim_\mathrm{H} K_f\leq \underline{\dim}_\mathrm{B} K_f$). 
The following result shows that the answer to this question is negative:

\begin{theorem} \label{mainThm:h_d_fail1} 
Theorem \ref{mainThm:ub_dvoretzky-hawkes}-(b) is not true under the assumption $\dim_\mathrm{H} K_f<1$, that is, there exist a probability density function $f$ with $m_f>0$ and $\dim_\mathrm{H} K_f<\overline{\dim}_\mathrm{B} K_f= 1$, and a sequence $\{\ell_n\}_{n \in \mathbb{N}}$ satisfying \eqref{mD} such that there is no $\mu_f$-Dvoretzky covering for $\mathbb{T}$. 
\end{theorem}

In Theorem \ref{mainThm:h_d_fail2}, Section \ref{examples}, we provide conditions on $\{\ell_n\}_{n \geq 1}$ under which Theorem \ref{mainThm:ub_dvoretzky-hawkes}-(b) is true if $\dim_\mathrm{H} K_f<1$.

It is clear from Theorem \ref{mainThm:h_d_fail1} that  condition \eqref{mD} is not always sufficient for covering. 
Nevertheless, we show that \eqref{mD} is in fact a ``Menshov generic" condition for $\mu_f$-Dvoretzky covering when $|K_f|=0$. 
Namely, 
when $|K_f|=0$ and $\{\ell_n\}_{n \geq 1}$ satisfies \eqref{mD}, Dvoretzky covering can be achieved by changing $f$ on a set of arbitrarily small Lebesgue measure. 
Note that the case $m_f=0$ is allowed in the following result. 

\begin{theorem}\label{mainThm:perturbation}
Let $f$ be a probability density function and $\{\ell_n\}_{n \in \mathbb{N}}$ be a decreasing sequence of positive numbers with $0<\ell_{n}<1$ satisfying \eqref{mD}. 
If $|K_f|=0$, then for every $\varepsilon >0$ one can find a density function $f_0$ so that $m_{f_0}=m_{f}$, $K_{f_0}=K_f$,  
\[
|\{x\in \mathbb{T}\colon f(x)\neq f_0(x)\}|\leq \varepsilon,
\]
and the circle is $\mu_{f_0}$-Dvoretzky covered for $\{\ell_n\}_{n \in \mathbb{N}}$.
\end{theorem}

{

}

In Section \ref{examples} we construct examples that complement the results above in various aspects. We also discuss some corollaries of our theorems.

\section*{Acknowledgements}
The authors would like to thank Tomas Persson, Jeffrey Steif, Grigori Karagulyan,  Aihua Fan and J\"org Schmeling for discussions on the manuscript. The first author is partially supported by Japan Society for the Promotion of Science (JSPS)
KAKENHI Grant Number 19K03558. The second author is supported by the Knut and Alice
Wallenberg foundation of Sweden (KAW) .

\thanks{}

\addtocontents{toc}{\protect\setcounter{tocdepth}{1}}
\section{Preliminaries} \label{preliminary}

\subsection{Some definitions} \label{definitions}

\subsubsection*{Essential infimum}

For a function $f\colon \mathbb{T}\to \mathbb{R}$ and $J\subset \mathbb{T}$ with $|J|>0$, define the \emph{essential infimum} of $f$ on $J$ by
\[
\mathrm{ess}\inf\nolimits_{J}f=\sup \{ a\in \mathbb{R}\colon f(x)\geq a\text{ for Lebesgue almost every }x\in J\}.
\]
We denote $m_f=\mathrm{ess}\inf_\mathbb{T}f$. 
Then the \emph{set of essential infimum points} is defined by
\[
K_f=\left\{ x\in \mathbb{T}\colon \lim_{n\to \infty}\mathrm{ess}\inf\nolimits_{B(x,1/n)}f=m_f\right\},
\]
where $B(x,r)$ denotes the ball centered at $x$ of radius $r$. 
For every Borel measurable function $f\colon \mathbb{T}\to \mathbb{R}$, the set $K_f$ is non-empty and compact (\cite{Fan-Karagulyan2021}*{Porposition 2.1}).

\subsubsection*{Hausdorff dimension and box dimension}

We recall briefly the notion of Hasudorff dimension and box dimension. 
See, for example, ~\cite{Falconer1990}.

Let $(X,d)$ be a metric space and $\alpha \geq 0$. 
For $Y\subset X$, set 
\begin{equation*}
\mathcal{H}^\alpha _\delta (Y)=\inf \left\{ \sum _{i\in \mathbb{N} }(\mathrm{diam}V_i)^{\alpha }\right\} ,
\end{equation*}
where the infimum is taken over all covers of $Y$ by a countable collection of subsets $V_i\subset X$ with $\mathrm{diam} V_i\leq \delta $. 
The \emph{$\alpha $-dimensional Hausdorff measure} of $Y$ is defined as
\begin{equation*}
\mathcal{H}^{\alpha }(Y)=\lim _{\delta \to 0}\mathcal{H}^\alpha _\delta (Y). 
\end{equation*} 
The limit exists and $\mathcal{H}^{\alpha }$ defines an outer measure. 
The \emph{Hausdorff dimension} $\dim _\mathrm{H}Y$ of $Y$ is defined by 
\begin{equation*}
\dim _\mathrm{H}Y=\sup \left\{ \alpha \colon \mathcal{H} ^{\alpha }(Y)=\infty \right\} =\inf \left\{ \alpha \colon \mathcal{H} ^{\alpha }(Y)=0\right\} .
\end{equation*}

For any non-empty bounded set $Z\subset X$ and $\delta \in (0,\infty)$, let $M(Z,\delta )$ be the smallest number of $\delta $-balls needed to cover $Z$. 
The \emph{lower} and \emph{upper box dimensions} of $Z$ are defined as
\[
\underline{\dim} _\mathrm{B}Z=\liminf_{\delta \to 0} \frac{\ln M(Z,\delta)}{-\ln \delta}
\]
and 
\[
\overline{\dim} _\mathrm{B}Z=\limsup_{\delta \to 0} \frac{\ln M(Z,\delta)}{-\ln \delta},
\]
respectively. 
If these are equal, then the common value is referred to as the \emph{box-counting} or \emph{box dimension} of $Z$ and denoted by $\dim _\mathrm{B}Z$. 
(There are several equivalent definitions of box dimension.) 
It is sometimes referred to as the \emph{Minkowski dimension}. 

It is known that $\dim _\mathrm{H}Z\leq \underline{\dim} _\mathrm{B}Z\leq \overline{\dim} _\mathrm{B}Z$, and these inequalities can be strict. 

\subsubsection*{Capacity} For $a>0$, define
\[
\Phi^{(a)}(t,s)=\exp \left\{a \sum_{n=1}^\infty (\ell_n - |t-s|)_+\right\},
\]
where $(d)_+=\max \{ d,0\}$. 
Let $F\subset \mathbb{T}$ be a compact set. 
We say $F$ has zero capacity with respect to $\Phi^{(a)}$ and denote by $\mathrm{Cap}_{a}(F)=0$ if
\[
\int_{F} \int_{F} \Phi^{(a)}(t,s)\, d\sigma(t)d\sigma(s)=\infty
\]
for every Borel probability measure $\sigma $ supported on $F$. 
Otherwise, we say $F$ has positive capacity and write $\mathrm{Cap}_{a}(F)>0$. Note that $\mathrm{Cap}_{1}$ coincides with the (classical) Kahane Capacity \eqref{cond:Kahane}.

\subsection{A sufficient condition for non-uniform Dvoretzky covering of the circle} \label{FK}

 In \cite{Fan-Karagulyan2021} the authors give a sufficient condition for covering when $m_f>0$ as follows:  
\begin{proposition}[\cite{Fan-Karagulyan2021}*{Proposition 5.2}] \label{cond:fk}
Assume $m_f>0$. 
If the conditions 
\begin{equation}\label{cond:shepp^a}
	 \text{for all } a >m_f,\quad \sum_{n=1}^\infty \frac{1}{n^2}e^{a(\ell_1+\dots + \ell_n)}=\infty ;
\end{equation}
	and
\begin{equation}\label{cond:kahane^m}
	    \mathrm{Cap}_{m_f}(K_f)=0,
\end{equation}
are fulfilled, then the circle is $\mu _f$-Dvoretzky covered.
\end{proposition}
It is not known if any of the conditions above is necessary for covering. In the next section we will prove the necessity of condition \ref{cond:shepp^a} for covering.

\addtocontents{toc}{\protect\setcounter{tocdepth}{2}}
\section{Proof of Theorem \ref{mainThm:ub_dvoretzky-hawkes}}


In this section, we prove the following result which gives a necessary condition of Shepp type for non-uniform Dvoretzky covering of $\mathbb{T}$. 

\begin{proposition}\label{mainProp:DtoS}
Assume for a sequence $\{\ell_n\}_{n\in \mathbb{N}}$ there is $\mu_f$-Dvoretzky covering for a probability density function $f$. 
Then the following condition is necessary for covering
\begin{equation*}
    \text{for all } a >m_f,\quad \sum_{n=1}^\infty \frac{1}{n^2}e^{a(\ell_1+\dots + \ell_n)}=\infty.
\end{equation*}
\end{proposition}

\subsection{Almost sure convergence for convolution operators and approximate identities} \label{discussion}

We now discuss the main idea of the proof of Proposition \ref{mainProp:DtoS}. 
First, recall the following local version of Billard's criterion from \cite{Fan-Karagulyan2021}: 

\begin{proposition}[Local Billard criterion]\cite{Fan-Karagulyan2021}*{Theorem 3.1} \label{prop:loc_B}
Let $\{\ell_n\}_{n\in \mathbb{N}}$ be a decreasing sequence of positive numbers with $0<\ell _n<1$, and $F \subset \mathbb{T}$ be a non-empty compact set. 
Suppose that $\mu$ is a probability measure on $\mathbb{T}$ such that
\begin{equation}\label{eq:cond-noncover1}
    \sup_{t \in F} \sum_{n=1}^\infty \mu(B(t, r_n))^2 <\infty ,
\end{equation}
where $r_n=\ell_n/2$. 
Then for $F$ there is no $\mu$-Dvoretzky covering if there exists a probability measure $\sigma$ supported on $F$ such that 
\begin{equation} \label{eq:cond-noncover2}
    \int_{F} \int_{F} \exp\left\{ \sum_{n=1}^\infty \mu(B(t, r_n)\cap B(s,r_n))\right\} \, d\sigma(t)d\sigma(s) <\infty.
\end{equation}
\end{proposition}

For $\ell_{n+1}< r\leq \ell_n$, let
\[
A_r(u)=\frac{\sum_{k=1}^\infty \mathbbm{1}_{I_{k,r}}(u+ \frac{r}{2})}{\sum_{k=1}^\infty (\ell_k-r)_+},
\]
where
\[
I_{k,r}=\left[\frac{-\ell_k+r}{2},\frac{\ell_k-r}{2}\right). 
\]
Note that
$
\int_\mathbb{T}\psi_{r}(u)\, du=1$
and for every $\delta>0$ we have that
\[
\lim_{r \rightarrow 0}\int_{-\delta}^{\delta}\psi_{r}(u)\, du=1.
\]
Thus $\psi_r$ is an \emph{approximate identity}. Note that $\psi_r(u)$ is symmetric around the point $u=-r/2$ and its mass is mostly concentrated at that point.
Observe now that if $s<t$ and $r=|t-s|$, then
\[
\frac{\sum_{n=1}^\infty \mu_f(B(t, r_n)\cap B(s,r_n))}{\sum_{n=1}^\infty (\ell_n-|t-s|)_+}=\int_\mathbb{T}f(s-t)\psi_r(t)\, dt = f\ast \psi_r(s).
\]

If one now has that $f \ast \psi_r \rightarrow f(s)$ for almost every $s$ as $r \rightarrow 0$, then for the set $F \subset \mathbb{T}$, where $f(s) \leq m_f + \varepsilon$, we will have by some additional work that
\begin{multline*}
     \int_{F} \int_{F} \exp \left\{ \sum_{n=1}^\infty \mu_f(B(t, r_n)\cap B(s,r_n))\right\} \, dtds \\ 
     \leq C\int_{F} \int_{F}\exp\left\{(m_f + \varepsilon)\left(\sum_{k=1}^\infty \ell_k-|t-s|\right)\right\}\, dtds.
\end{multline*}

The last estimate, in fact, implies Proposition \ref{mainProp:DtoS}. However, as it will be seen from the discussion below, the operators $f \ast \psi_r$ can exhibit rather erratic behaviour and for general sequences $\{\ell_n\}_{n \in \mathbb{N}}$ it is not easy to establish for which sequences $r_k \rightarrow 0$ does one have $f \ast \psi_{r_k}(s)\rightarrow f$. Hence, it is not clear how to estimate the value of $f \ast \psi_r$. 

We remark that 
 there is a large literature devoted to the study of convergence properties of convolution operators $f\ast \psi_r(s)$ and $\psi_r(s)$ in the scope of kernels studied in these papers. We refer the reader to  \cite{Kostyukovsky-Olevskii2004} for an overview of results on convergence of convolution operators both for absalutely continuous and singular measures and also to \cites{Karagulyan-Safaryan2017, Karagulyan-Safaryan2015} for some further references and recent results. For discrete measures, we recall the classical result of Bourgain \cite{Bourgain1988}, where divergence is shown for discrete measures. Therefore, understanding the convergence properties of $f \ast \psi_r$ is a problem of independent interest: 

\begin{proposition}\label{qtn1}
Assume that
\begin{equation}\label{L11}
\limsup_{n \rightarrow \infty}\frac{n \ell_n}{ \ell_1 + \cdots + \ell_n}<1.
\end{equation}
Then for every $f \in L^1$ we have
$$
f \ast \psi_r(s)\rightarrow f(s), \text{ as }r \rightarrow 0,
$$ 
for almost every $s \in \mathbb{T}$.
\end{proposition}
\begin{proof}
In \cite{Karagulyan-Safaryan2015}*{Theorem 2.1}, it is shown that when $\limsup _{r \rightarrow 0} (r/2)\left\|\psi_{r}\right\|_{\infty}<\infty$ ($\lambda(r)=r/2$ is the location of the tip of $\psi_r$), then $f \ast \psi_r (s)$ converges to $f(s)$ almost surely. In our case this is equivalent to the existence of $C>0$ so that for all $\ell_{n+1}< r\leq \ell_n$ we have
\[
\frac{(r/2) n}{\sum_{k=1}^n (\ell_k - r)}<C.
\]
Since $r \leq \ell_n$, a quick computation shows that the above is equivalent to
\[
\limsup_{n \rightarrow \infty}\frac{n \ell_n}{\sum_{k=1}^n \ell_k}<1.
\]
\end{proof}

We remark that the necessity of condition \eqref{L11} is not completely trivial. In \cite{Karagulyan-Safaryan2017}*{ Theorem 1.1}, the authors show that the operators $f\ast\psi_r  (s)$ can be divergent almost surely for a characteristic function $f$ if
$$
\limsup _{\delta \rightarrow 0} \liminf _{r \rightarrow 1} \int_{\frac{r}{2} -\frac{\delta r}{2}}^{\frac{r}{2} +\frac{\delta r}{2}}\psi_r(s)\, ds>\frac{1}{2}.
$$

However this condition amounts to the assumption
$$
\lim_{n \to \infty}\frac{n \ell_n}{\ell_1+ \dots + \ell_n}=1.
$$
which is too strong and it fails if $\ell_n$ converge to zero. Thus, we see that the general problem of almost everywhere convergence properties of the operator $f \ast \psi_r$ is rather non-trivial. In our case one will in fact have
$$
\limsup _{\delta \rightarrow 0} \limsup_{r \rightarrow 1} \int_{\frac{r}{2} -\frac{\delta r}{2}}^{\frac{r}{2} +\frac{\delta r}{2}}\psi_r(s)\, ds=1.
$$
Therefore, we have the following open question:

\begin{question}
Assume for some $\{\ell_n\}$ we have 
$$
\sum_{n=1}^\infty \ell_n^2<\infty.
$$
Assume further that \eqref{L11} fails. Does there exists $f \in L^1$, so that $f \ast \psi_r(s)$ does not converge to $f$, as $r \rightarrow 0$?
\end{question}

We remark that although we are unable to show convergence for all sequences $r_k \rightarrow 0$, however the discussion above gives us some idea on where the volatile behavior of $f \ast \psi_r$ lies and how to treat the problem. The idea is to approximate $\psi_r$ with another kernel which is easier to analyze.

\subsubsection{Main Lemma and Proof of Proposition \ref{mainProp:DtoS}} \label{aux}

We now introduce an approximation $\phi_r$ of the kernel $\psi_r$ which allows us to isolate the set of $r$ where the behaviour of $\psi_r$ becomes too volatile, i.e. when there is huge concentration of the sequence $\{\ell_n\}$. 

Fix some $b<1$ very close to $1$ and for each $q \geq 0$ define
\[
\nu_q = \#\{n\in \mathbb{N} \colon b^{q+1}<\ell_n\leq b^q\}.
\]
Then define the sequence 
\[
\ell_n'=b^{q+1}
\]
if $\nu_0+\dots + \nu_q < n < \nu_0+\dots + \nu_{q+1}$. 
We also set $r_n'=\ell_n'/2$.
One can easily check that
\[
b\ell'_n\leq \ell_n \leq \ell'_n.
\]
We now replace $\ell_n$ with the sequence $\ell_n'$. 
Note that if there is $\mu_f$-Dvoretzky covering for the sequence $\{\ell_n\}_{n\in \mathbb{N}}$, then there will also be $\mu_f$-Dvoretzky covering for the sequence $\{\ell'_n\}_{n\in \mathbb{N}}$.
For $\varepsilon>0$ small we define the set
\[
F =\left\{ x\in \mathbb{T} \colon f(x)\leq m_f + \varepsilon\right\}.
\]
Then $|F|>0$. 
Next define a function
\begin{equation}\label{times}
g(x)=\begin{cases} 
f(x)  & \text{if}\quad x \notin F, \\
m_f + \varepsilon & \text{if}\quad x \in F.
\end{cases}
\end{equation}
Notice that $m_g=m_f +\varepsilon$. 

For all $s \in F$, let
\[
G_{n}(s):=\frac{\mu_{g}\left(B\left(s, r_{n}'\right)\right)}{\ell_{n}'}=\frac{1}{\ell_{n}'} \int_{s-r_{n}'}^{s+r_{n}'} g(x) \, d x. 
\]
By the Lebesgue differentiation theorem, one has $\lim _{n\to \infty }G_n(s)=g(s)$ for Lebesgue almost every $s\in F$. 
Then, by Egorov's theorem, there exists a compact set $F_0\subset F$ with $|F_0|>|F|/2>0$ such that $G_n$ converges to $g$ uniformly on $F_0$ as $n\to \infty$. 
Hence, there is an integer $N_0=N_0(\varepsilon)>1$ such that for every $n\geq N_0$ one has
\[
|G_n(s)-g(s)|<\varepsilon 
\]
for every $s\in F_0$. 
It follows that for $n\geq N_0$ and $s\in F_0$ one has 
\begin{equation} \label{ineq:F_0ball}
\mu_{g}\left(B(s, r_{n}')\right)<\left(g(s)+\varepsilon \right)\ell _n'=\left(m_f+2\varepsilon\right)\ell _n'.
\end{equation}

Let $\phi_r$ be the kernel defined in the preceding section associated with the sequence $\{\ell_n'\}_{n \in \mathbb{N}}$, or more specifically,  
\[
\varphi _r(u)=\frac{\sum_{n=1}^\infty \mathbbm{1}_{I_{n,r}'}\left(u+\frac{r}{2}\right)}{\sum_{n=1}^\infty (\ell_n'-r)_+} ,
\]
where 
\[
I_{k,r}'=\left[\frac{-\ell_n'+r}{2},\frac{\ell_n'-r}{2}\right) =\left[-r_n'+\frac{r}{2},r_n'-\frac{r}{2}\right) . 
\]
We then set $r=|t-s|$. 

\begin{lemma}\label{Lmm-main}
Let $c\in (b,1)$ and $\delta \in (0,1)$. 
Then for every $\varepsilon_0>0$ there exists $r_0>0$ so that 
\begin{enumerate}
    \item[1)] For all $r \leq r_0$, with $b^{K+1}< r \leq cb^{K}$, and for almost every $s \in F_0$ we have that 
    \[
    g\ast \phi_r(s) \leq  m_f + \varepsilon_0.
    \]
    
    \item[2)] For all $r \leq r_0$, with $b^{K+1}<r\leq  b^K$ and $K$ such that \begin{equation}\label{kernel} \frac{(\nu_0+\dots + \nu_K)b^K}{\sum_{k=0}^K \nu_k b^k}\leq\delta, 
    \end{equation} 
    then for almost every $s \in F_0$ 
    \[ 
    g\ast \phi_r(s) \leq m_f+ \varepsilon_0. 
    \]
    
    \item[3)] Assume that for some $a>0$ we have 
    \begin{equation} \label{contra-shepp^a}
    \sum_{n=1}^\infty \frac{1}{n^2}e^{a(\ell_1 + \dots + \ell_n)}<\infty. 
    \end{equation}
    Then the constants $c,\delta$ above, can be chosen in such a way that one can find a subset $F_1 \subset F_0$ so that for all $r \leq r_0$, with $cb^{K} < r \leq b^{K}$ and $K$ not full-filing \ref{kernel}, we have for a.e. $s \in F_1$ 
    \[ 
    g\ast \phi_r(s) \leq \frac{\frac{K}{2} \ln \frac{1}{b}}{\sum_{k=1}^\infty \nu_k(b^k-r)_+}. 
    \]
\end{enumerate} 
\end{lemma}

\begin{remark}
One can compare condition \eqref{kernel} with \eqref{L11} and see that $\limsup$ in \eqref{L11} is achieved at times $r=b^k$, $k=1,2, \dots$. Using  ideas similar to those in the proof of Lemma \ref{Lmm-main} one can in fact show that when $r \rightarrow 0$, subject to the conditions in $1)$ or $2)$, we have that $g\ast \phi_r(s) \rightarrow g(s)$, for almost every $s \in \mathbb{T}$ and any $g \in L^1$. We believe that one can then transfer the same property to the kernel $\psi_r$. However, we prove Lemma 1 under weaker assumptions since it is sufficient for our purposes. We do not know if there is convergence for $r$ in case 3).
\end{remark}

We postpone the proof of this lemma and conclude  the proof of Proposition \ref{mainProp:DtoS}. 
Note that it now follows from Lemma \ref{Lmm-main} that for arbitrary $\epsilon>0$ and $c$ close to $1$ we have 
\begin{align*}
    &\int \int_{\{ (t,s)\in F_0 \times F_0\colon b^{K+1}<|t-s|\leq cb^K\}}\exp \left\{\sum_{n=1}^\infty \mu_g(B(t, r_n')\cap B(s,r_n'))\right\}\, dtds \\
    &\quad \leq C_{\varepsilon, c} \int \int_{\{ (t,s)\in F_0 \times F_0\colon b^{K+1}<|t-s|\leq cb^K\}}\exp\left\{(m_g + \epsilon)\sum_{k=0}^\infty \nu_k (b^k - |t-s|)_+\right\}\, dtds,
\end{align*}
for some $C_{\varepsilon, c}\in (0,\infty)$.
And also
\begin{equation} 
    \begin{split}
    &\iint_{\{(t,s)\in F_0 \times F_0\colon b^{K+1}<|t-s|\leq b^K, \eqref{kernel} \text{ holds for }K\}}\exp\left\{\sum_{n=1}^\infty \mu_g(B(t, r_n')\cap B(s,r_n'))\right\}\, dtds \\
    &\leq C_{\varepsilon, \delta} \iint_{\{(t,s)\in F_0 \times F_0\colon b^{K+1}<|t-s|\leq b^K, \eqref{kernel} \text{ holds for }K\}}\exp\left\{(m_g + \epsilon)\sum_{k=0}^\infty \nu_k (b^k - |t-s|)_+\right\}\, dtds
    \end{split}
\end{equation}
for some $C_{\varepsilon, \delta}\in (0,\infty)$.

For each $t\in F_1$ and $K\in \mathbb{N}$, let $F_1(t,K)=\{ s\in F_1\colon cb^K<|t-s|\leq b^K\} $, and  $F_1(t)=\cup _{K\notin \mathcal{G}}F_1(t,K)$, where $\mathcal{G}$ is the set of $K\in \mathbb{N}$ such that the condition \eqref{kernel} is fulfilled for some $\delta >0$.  
Then by the Fubini theorem, 
\begin{align*}
    &\iint_{\{(t,s)\in F_1 \times F_1\colon cb^{K}<|t-s|\leq b^K, K\not\in \mathcal{G}\}} \exp\left\{ \sum_{n=1}^\infty \mu_g(B(t, r_n')\cap B(s,r_n'))\right\} \, dtds \\
    &\quad \leq \int_{F_1}\left[ \int_{F_1(t)} \exp\left\{ \sum_{n=1}^\infty \mu_g(B(t, r_n')\cap B(s,r_n'))\right\} \, ds\right] dt \\
    &\quad \leq \sum _{K\notin \mathcal{G}} \int _{F_1}\left[ \int_{F_1(t,K)}\exp  \left( \frac{K}{2} \ln \frac{1}{b}\right) \, ds \right]\, dt \\
    &\quad \leq 2(1-c)|F_1|\sum_{K} b^{K}\exp \left( \frac{K}{2} \ln \frac{1}{b}\right) =2(1-c)|F_1|\sum_{K}\exp \left( \frac{K}{2}\ln \frac{1}{b}-K \ln \frac{1}{b}\right) <\infty
\end{align*}
as $|F_1(t,K)|\leq 2(1-c)b^K$. 

Summarizing the above, one has 
\begin{equation} \label{ineq:conclusion}
    \begin{split}
    &\iint_{F_1\times F_1} \exp\left\{\sum_{n=1}^\infty \mu_g(B(t, r_n')\cap B(s,r_n'))\right\} \, dtds \\
    &\quad \leq C \iint_{F_1\times F_1} \exp\left\{ (m_g + \varepsilon)\sum_{n=1}^\infty (\ell_n'-|t-s|)_+\right\} \, dtds 
    \end{split}
\end{equation}
for some $C\in (0,\infty )$ if \eqref{contra-shepp^a} is assumed. 
(Obviously if no such $a>0$ exists then the result in Proposition \ref{mainProp:DtoS} follows.) 
On the other hand, under the same assumption, we have 
\[
\sup _{s\in F_1}\sum _{n=1}^\infty \mu_g(B(s,r_n'))^2\leq (m_f+2\varepsilon)^2\sum _{n=1}^\infty (\ell_n')^2\leq \left( \frac{m_f+2\varepsilon}{b}\right)^2\sum _{n=1}^\infty \ell_n^2<\infty
\]
by \eqref{ineq:F_0ball} and \cite{Fan-Karagulyan2021}*{Lemma 3.2}. 
Then by Proposition \ref{prop:loc_B} (local Billard criterion) and the fact $f\leq g$ on $\mathbb{T}$ one should have
\begin{equation*}
    \iint_{F_1\times F_1} \exp\left\{ \sum_{n=1}^\infty \mu_g(B(t, r_n')\cap B(s,r_n'))\right\} \, dtds=\infty.
\end{equation*}
Therefore it follows from \eqref{ineq:conclusion} that 
\[
\iint_{F_1\times F_1} \exp\left\{ (m_g + \varepsilon)\sum_{n=1}^\infty (\ell_n'-|t-s|)_+\right\} \, dtds=\infty.
\]
Note that above we can replace $F_1$ with $\mathbb{T}$. Then, a change of coordinates in the exponent yields
\[
\iint_{\mathbb{T}\times \mathbb{T}} \exp\left\{ \sum_{n=1}^\infty ((m_g + \varepsilon)\ell_n'-|t-s|)_+\right\} \, dtds=\infty.
\]
Which, in view of \cite{Kahane1985}*{Theorem 2 (page 150)}, is equivalent to
\[
\sum_{n=1}^\infty \frac{1}{n^2}e^{(m_g+\varepsilon)(\ell'_1 + \dots + \ell'_n)}=\infty.
\]
But note that
\[
\sum_{n=1}^\infty \frac{1}{n^2}e^{(m_g+\varepsilon)(\ell'_1 + \dots + \ell'_n)}\leq \sum_{n=1}^\infty \frac{1}{n^2}e^{\frac{1}{b}(m_f+2\varepsilon)(\ell_1 + \dots + \ell_n)}=\infty
\]
as $b\ell'_n\leq \ell_n$ and $m_g = m_f + \varepsilon$. 
Since $b\in (0,1)$ is a number arbitrarily close to 1 and $\varepsilon>0$ was arbitrary, one arrives at 
\[
\sum_{n=1}^\infty \frac{1}{n^2} e^{a(\ell_1 + \dots + \ell_n)}=\infty
\]
for every $a>m_f$. 
Proposition \ref{mainProp:DtoS} is proven. 

\subsection{Proof of Lemma \ref{Lmm-main}} 

We first prove a lemma.

\begin{lemma} 
For every $n\geq N_0$, $s\in F_0$, and $t\in \mathbb{T}$ one has 
\begin{equation} \label{ineq:intersection}
    \mu _g(B(t, r_n')\cap B(s,r_n'))\leq m_g(\ell_n' - |t-s|)_+ +\varepsilon \ell_n'. 
\end{equation}
\end{lemma}
\begin{proof}
Note that for all $s,t\in \mathbb{T}$, one has 
\begin{equation} \label{ineq2:intersection}
\mu _g(B(t, r_n')\cap B(s,r_n'))\leq m_g(\ell_n' - |t-s|)_+ +(\mu _g(B(s,r_n'))-m_g\ell_n').  
\end{equation}
Indeed, if $|t-s|>\ell _n'$, then this inequality reads as
\[
\mu _g(B(t, r_n')\cap B(s,r_n'))=0\leq 0+\mu _g(B(s,r_n'))-m_g\ell _n',
\]
and this will follow as 
\[
\mu _g(B(s,r_n'))=\int _{B(s,r_n')}g\, dx\geq m_g\ell _n'.
\]
For $s,t\in \mathbb{T}$ such that $|t-s|\leq \ell _n'$, the inequality \eqref{ineq2:intersection} reads as
\[
\mu _g(B(t, r_n')\cap B(s,r_n'))\leq -m_g |t-s|+ \mu _g(B(s,r_n')),
\]
and this will follow since
\begin{align*}
    \mu _g(B(s,r_n'))
    &=\mu _g(B(t, r_n')\cap B(s,r_n'))+\mu _g(B(s,r_n')\setminus B(t, r_n')) \\
    &\geq \mu _g(B(t, r_n')\cap B(s,r_n'))+m_g|t-s|.
\end{align*}
Hence \eqref{ineq2:intersection} is obtained. 
It follows from \eqref{ineq:F_0ball} and \eqref{ineq2:intersection} that for $n\geq N_0$, $s\in F_0$, and $t\in \mathbb{T}$ one has 
\begin{align*} 
    \mu _g(B(t, r_n')\cap B(s,r_n'))
    &\leq m_g(\ell_n' - |t-s|)_+ +(m_f+2\varepsilon -m_g)\ell_n'=m_g(\ell_n' - |t-s|)_+ +\varepsilon \ell_n'. 
\end{align*}
Lemma is proved. 
\end{proof}

\begin{proof}[Proof of Lemma \ref{Lmm-main}]

Let $N>N_0$, and $s\in F_0$, $t\in \mathbb{T}$ such that $\ell _{N+1}<|t-s|\leq \ell _N$. 
Then 
\[
\sum_{n=1}^\infty \mu_g(B(t, r_n')\cap B(s,r_n')) =
\sum_{n=1}^N \mu_g(B(t, r_n')\cap B(s,r_n'))
\]
as $B(t, r_n')\cap B(s,r_n')=\emptyset $ for $n\geq N+1$, and by \eqref{ineq:intersection} one has
\begin{align*}
    &\sum_{n=1}^N \mu_g(B(t, r_n')\cap B(s,r_n')) \\
    &=\sum_{n=1}^{N_0-1} \mu_g(B(t, r_n')\cap B(s,r_n')) +\sum_{n=N_0}^{N} \mu_g(B(t, r_n')\cap B(s,r_n')) \\
    &\leq \sum_{n=1}^{N_0-1} \mu_g(B(t, r_n')\cap B(s,r_n')) +\sum_{n=N_0}^{N} \left\{ m_g(\ell_n' - |t-s|)_+ +\varepsilon \ell_n'\right\} \\
    &\leq \sum_{n=1}^{N_0-1} \left\{ m_g(\ell_n' - |t-s|)_+ +\varepsilon \ell_n'\right\} +\sum_{n=N_0}^{N} \left\{ m_g(\ell_n' - |t-s|)_+ +\varepsilon \ell_n'\right\} +R_0 \\
    &=\sum_{n=1}^{N} \left\{ m_g(\ell_n' - |t-s|)_+ +\varepsilon \ell_n'\right\} +R_0,
\end{align*}
where we set $R_0=\sum_{n=1}^{N_0-1} \mu_g(B(t, r_n')\cap B(s,r_n'))$. 
Note that one can write
\[
N=\nu_0+\dots +\nu_{K-1}+u
\]
for some $K=K_N\in \mathbb{N}$ and $u \in \{ 1,\dots ,\nu_{K}\}$. 
It follows that 
\begin{align*}
    &\sum_{n=1}^\infty \mu_g(B(t, r_n')\cap B(s,r_n')) \\
    &=\sum_{n=1}^{\nu_0+\dots +\nu_{K-1}} \mu_g(B(t, r_n')\cap B(s,r_n'))+\sum_{n=\nu_0+\dots +\nu_{K-1}+1}^{N} \mu_g(B(t, r_n')\cap B(s,r_n')) \\
    &\leq m_g\sum_{k=0}^{K-1} \nu_k(b^k - |t-s|)_+ +\sum_{k=0}^{K-1} \varepsilon \nu_k b^k +\sum_{n=\nu_0+\dots +\nu_{K-1}+1}^{N} \mu_g(B(t, r_n')\cap B(s,r_n'))  +R_0.
\end{align*}
Note that for every $k\in \{ 0,1,\dots ,K-1\}$, one can see that 
\begin{equation*}
    b^k(1-b)\leq b^k-b^K\leq b^k-|t-s|
\end{equation*}
as $b^{K-k}\leq b<1$, and hence one has
\[
\sum_{k=0}^{K-1} \varepsilon \nu_kb^k\leq \sum_{k=0}^{K-1} \varepsilon \nu_k\frac{b^k-|t-s|}{1-b}.
\]
It follows that 
\begin{equation} \label{ineq:principal_sum}
    \begin{split}
    \sum_{n=1}^\infty \mu_g(B(t, r_n')\cap B(s,r_n')) 
    &\leq \left( m_g+\frac{\varepsilon }{1-b}\right) \sum_{k=0}^{K-1} \nu_k(b^k - |t-s|)_+ \\
    &\quad +\sum_{n=\nu_0+\dots +\nu_{K-1}+1}^{N} \mu_g(B(t, r_n')\cap B(s,r_n'))  +R_0.
    \end{split}
\end{equation}

We now estimate the right-hand side of \eqref{ineq:principal_sum}  separating the cases 1), 2), and 3).  

\noindent 
{\bf Case 1)} 
Note that $b^K(1-c) \leq b^K-r$ as $r=|t-s|\leq c b^K$. 
Hence
\begin{align*}
    \sum_{n=\nu_0+\dots +\nu_{K-1}+1}^{N} \mu_g(B(t, r_n')\cap B(s,r_n')) 
    &\leq \nu_K\left\{ m_g(b^K - r)_+ +\varepsilon b^K\right\} \\
    &\leq \nu_Km_g(b^K - r)_+ +\nu _K\frac{\varepsilon}{1-c} (b^K - r)_+ \\
    &= \nu_K\left( m_g+\frac{\varepsilon }{1-c}\right) (b^K-r)_+.
\end{align*}
It follows that 
\begin{align*}
    \sum_{n=1}^\infty \mu_g(B(t, r_n')\cap B(s,r_n')) 
    &\leq \left( m_g+\frac{\varepsilon }{1-b}\right) \sum_{k=0}^{K-1} \nu_k(b^k - r)_+ \\
    &\quad +\left( m_g+\frac{\varepsilon }{1-c} \right) \nu_K(b^K-r)+R_0\\
    &\leq \left( m_g+\frac{\varepsilon }{1-b}+\frac{\varepsilon }{1-c}\right) \sum_{k=0}^{K} \nu_k(b^k - r)_+ +R_0 \\
    &= \left( m_g+\frac{\varepsilon }{1-b}+\frac{\varepsilon }{1-c}\right) \sum_{k=0}^{\infty} \nu_k(b^k - r)_+ +R_0.
\end{align*}
as $(b^k - r)_+=0$ for $k>K$. 
Hence
\[
g \ast \phi_r=\frac{\sum_{n=1}^\infty \mu_g(B(t, r_n')\cap B(s,r_n'))}{\sum_{k=0}^{\infty} \nu_k(b^k - r)_+}\leq m_g+ \varepsilon_0.
\]

\noindent 
{\bf Case 2)} 
By assumption \eqref{kernel}, we have
\begin{equation}\label{inter}
\sum_{k=0}^K \nu_k(b^k-r)_+ \ge \sum_{k=0}^K \nu_k(b^k-b^K)_+ \geq (1- \delta)\sum_{k=0}^K \nu_k b^k. 
\end{equation}
Since 
\begin{align*}
    \sum_{n=\nu_0+\dots +\nu_{K-1}+1}^{N} \mu_g(B(t, r_n')\cap B(s,r_n')) &\leq \sum_{n=\nu_0+\dots +\nu_{K-1}+1}^{N}(b^K-r)_++\varepsilon b^K \\
    &=\nu_K(b^K-r)_++\varepsilon \nu_K,
\end{align*}
it follows from \eqref{ineq:principal_sum} and \eqref{inter} that
\begin{align*}
    \sum_{n=1}^\infty \mu_g(B(t, r_n')\cap B(s,r_n')) 
    &\leq m_g\sum_{k=0}^K \nu _k(b^k-r)_+ + \sum_{k=0}^K \varepsilon \nu_k b^k+R_0 \\
    &\leq \left( m_g+\frac{\varepsilon}{1-\delta } \right) \sum_{k=0}^K \nu_k(b^k-r)_++R_0 \\
    &= \left( m_g+\frac{\varepsilon}{1-\delta } \right) \sum_{k=0}^\infty \nu_k(b^k-r)_++R_0. 
\end{align*}
Hence
\[
g \ast \phi_r(s) \leq m_g + \varepsilon_0.
\]

\noindent 
{\bf Case 3)} We now take care of terms for which $c b^{K}<|t-s|\leq b^K$ and \eqref{kernel} fails. 
One has 
\begin{align*}
    \sum_{n=\nu_0+\dots +\nu_{K-1}+1}^{N} \mu_g(B(t, r_n')\cap B(s,r_n'))
    &=\sum_{n=\nu_0+\dots +\nu_{K-1}+1}^{N} \int _{B\left( \frac{t+s}{2},\frac{b^K-|t-s|}{2}\right)}g\,dx \\
    &\leq \nu _K\int _{B\left( \frac{t+s}{2},\frac{b^K-|t-s|}{2}\right)}g\,dx,
\end{align*}
and thus 
\begin{equation} \label{estm-2}
    \begin{split}
    \sum_{n=1}^\infty \mu_g(B(t, r_n')\cap B(s,r_n')) 
    &\leq \left( m_g+\frac{\varepsilon}{1-b} \right) \sum_{k=0}^{K-1} \nu_k (b^k-|t-s|)_+ \\
    &\quad +\nu _K\int _{B\left( \frac{t+s}{2},\frac{b^K-|t-s|}{2}\right)}g\,dx+R_0.
    \end{split}
\end{equation}

\begin{sublemma} \label{lebesgue}
Let $\gamma_n\in (0,1]$ and $a_n\searrow 0$ as $n\to \infty$. 
Then for every $h \in L^1(\mathbb{T})$ we have
\[
\lim_{n\to \infty} \frac{1}{a_n \gamma _n} \int_{[s+a_n(1-\gamma_n),s+a_n]}h\, dx = h(s),
\]
for Lebesgue almost every $s \in \mathbb{T}$.
\end{sublemma}
\begin{proof}
By Lebesgue's differentiation theorem, for Lebesgue almost every $s \in \mathbb{T}$ we have
\[
\int_{[s,s+a_n]}h\, d x=(h(s) + o(1))a_n, \quad \int_{[s,s+a_n (1-\gamma_n)]}h\, d x=(h(s)+o(1))a_n(1-\gamma_n)
\]
as $n\to \infty $. 
Hence
\[
\int_{[s+a_n(1-\gamma_n),s+a_n]}h\, d x=(h(s)+o(1))a_n - (h(s)+o(1))a_n(1-\gamma_n)=h(s)a_n \gamma_n + o(1)
\]
as $n\to \infty $. 
\end{proof}

For the integral in the right-hand side of \eqref{estm-2}, it follows from Sublemma \ref{lebesgue} that 
\[
\lim_{K\to \infty} \frac{1}{(b^K-|t-s|)} \int _{B\left( \frac{t+s}{2},\frac{b^K-|t-s|}{2}\right)}g\,dx= g(s)=m_g
\]
for Lebesgue almost every $s\in F_0$. 
Here recall that $g(s)=m_g$ for $s\in F$. 
Using Egoroff's theorem, one has a compact set $F_1\subset F_0$ with $|F_1|>|F_0|/2>0$ such that the sequence of averages converges to $g$ uniformly on $F_1$ as $K\to \infty$. 
It follows that there is an integer $K_1>1$ such that for every $K\geq K_1$ and $s\in F_1$ one has 
\[
\int _{B\left( \frac{t+s}{2},\frac{b^K-|t-s|}{2}\right)}g\,dx\leq 2m_g(b^K-|t-s|).
\]
Since $b^K-|t-s|\leq b^K-cb^K$, we have 
\begin{equation} \label{ineq:integral}
    \begin{split}
    \sum_{n=1}^\infty \mu_g(B(t, r_n')\cap B(s,r_n')) 
    &\leq \left( m_g+\frac{\varepsilon}{1-b} \right) \sum_{k=0}^{K-1} \nu_k (b^k-|t-s|)_+ \\
    &\quad +2m_g\nu _Kb^K(1-c)+R_0
    \end{split}
\end{equation}
for every $s\in F_1$.

Let 
\[
\lambda_{K} = \frac{\nu_K}{\sum_{k=0}^{K-1}\nu_k}.
\]

\begin{sublemma} \label{lambda}
Assume for a sub-sequence $\{K_q\}_{q \geq 1}$ we have that
\[
\lim_{q\rightarrow \infty}\frac{(\nu_0+\dots +\nu_{K_q})b^{K_q}}{\sum_{k=0}^{K_q} \nu_k b^k}=1.
\]
Then we have $\lambda_{K_q} \to \infty$, as $q \rightarrow \infty$ and 
\[
\frac{\nu_{K_q} b^{K_q}}{\sum_{k=0}^{{K_q}-1} \nu_k b^k}\geq (1-o(1))\left( \frac{1}{\lambda _{K_q}} +o(1)\right) ^{-1}
\]
as $q \rightarrow \infty$. 
\end{sublemma}

\begin{sublemma} \label{nu}
Assume \eqref{contra-shepp^a}. 
Assume further that $\lambda_K\geq 1$ for some $K\in \mathbb{N}$, then there exists $C_0>0$, such that
\[
\nu_{K} b^{K} \leq C_0 K.
\]
\end{sublemma}

We postpone the proofs of the sublemmas above to Section \ref{proof:two lemmas} and proceed the argument. 
Given a small $\zeta >0$, by Sublemmas \ref{lambda} and \ref{nu}, one has
\begin{equation} \label{ineq:finite sum}
    \sum_{k=0}^{K-1} \nu_k (b^k-|t-s|)_+\leq \sum_{k=0}^{K-1} \nu_k b^k\leq \frac{1}{1-\zeta }\left( \frac{1}{\lambda _K} +\zeta \right) \nu_Kb^{K}\leq 2\left( \frac{1}{\lambda _K} +\zeta \right) C_0K. 
\end{equation}
It follows from \eqref{ineq:integral} and \eqref{ineq:finite sum} that
\begin{align*}
    \sum_{n=1}^\infty \mu_g(B(t, r_n')\cap B(s,r_n')) 
    &\leq 2C_0\left\{ \left( m_g+\frac{\varepsilon}{1-b} \right)\left( \frac{1}{\lambda _K} +\zeta \right) +m_g(1-c)\right\} K+R_0 \\
    &\leq C_1\left\{ \left( \frac{1}{\lambda _K} +\zeta \right) +(1-c)\right\} K+R_0,
\end{align*}
where we let $C_1=2C_0(m_g+(\varepsilon /(1-b)))>0$. 
Hence for large $K$, one has 
\[
\sum_{n=1}^\infty \mu_g(B(t, r_n')\cap B(s,r_n'))\leq C_2\left\{ \left( \frac{1}{\lambda _K} +\zeta \right) +(1-c)\right\} K
\]
for some $C_2>0$. 
We now take $\zeta >0$ small, $\lambda_K>1$ large by Sublemma \ref{lambda}, and $c\in (0,1)$ so close to $1$ that (recall that $b$ is a fixed number $b\in (0,1)$ arbitrarily close to 1)
\[
C_2\left\{ \left( \frac{1}{\lambda_K} +\zeta \right) +(1-c)\right\} \leq \frac{1}{2}\ln \frac{1}{b}.
\]
Lemma \ref{Lmm-main} is obtained. 
\end{proof}

\subsubsection{Proof of Subemma \ref{lambda} and Sublemma \ref{nu}} \label{proof:two lemmas}

\begin{proof}[Proof of Sublemma \ref{lambda}]
Observe that the function
\[
u\mapsto \beta _K(u)=\frac{(\nu_0+\cdots +\nu_{K-1}+u) b^K}{ \sum_{k=0}^{K-1}\nu_k b^K + u b^K}
\]
is increasing in $u$, hence $\beta _K|_{\{1, \ldots, \nu _{K}\}}$ will reach its maximum value for $u=\nu _K$.
Note that each $N\in \mathbb{N}$ can be written as 
\[
N=\nu _0+\dots +\nu _{K-1}+u
\]
for some $K \in \mathbb{N}$ and $u \in\left\{1, \ldots, \nu _{K}\right\}$, and for such an $N$ one has $\ell_{N}^{\prime}=b^{K}$. 
Hence
\begin{align*}
    \frac{N \ell_{N}^{\prime}}{\ell_{1}^{\prime}+\cdots+\ell_{N}^{\prime}} 
    &=\frac{\left(\nu_{0}+\nu_{1}+\cdots+\nu_{K-1}+u\right) b^{K}}{\nu_{0}+\nu_{1} b+\cdots+\nu_{K-1} b^{K-1}+u b^{K}} \\
    & \leq \frac{\left(\nu_{0}+\nu_{1}+\cdots+\nu_{K-1}+u\right) b^{K}}{\left(\nu_{0}+\nu_{1}+\cdots+\nu_{q-1}\right) b^{K-1}+u b^{K}} \\
    &=\frac{\left(\nu_{0}+\cdots+\nu_{K-1}+u\right) b}{\left(\nu_{0}+\cdots+\nu_{K-1}\right)+u b} \\
    & \leq \frac{\left(\nu_{0}+\cdots+\nu_{K-1}+\lambda_{K}\left(\nu_{0}+\cdots+\nu_{K-1}\right)\right) b}{\left(\nu_{0}+\nu_{1}+\cdots+\nu_{K-1}\right)+\lambda_{K}\left(\nu_{0}+\cdots+\nu_{K-1}\right) b} =\frac{\left(1+\lambda_{K}\right) b}{1+\lambda_{K} b}<1
\end{align*}
since $b<1$. 
By assumption, there exists a sequence $\{ \zeta_{q}\}$, with $\zeta_{q}\to 0$ as $q\to \infty$, such that 
\[
\frac{(\nu_0+\cdots +\nu_{K_q}) b^{K_q}}{ \sum_{k=0}^{K_q}\nu_k b^{k}}>1-\zeta_{q}
\]
for every $q\in \mathbb{N}$. 
It follows from above that 
\begin{equation} \label{ineq:lambda_nu_zeta}
    \frac{\left(1+\lambda_{K_q}\right) b}{1+\lambda_{K_q} b}\geq\frac{(\nu_0+\cdots +\nu_{K_q}) b^{K_q}}{ \sum_{k=0}^{K_q}\nu_k b^{k}}>1-\zeta_{q},
\end{equation}
or equivalently
\[
\lambda_{K_q}>\frac{1-\zeta_{q}-b}{\zeta_q b}.
\]
Thus one has $\lambda_{K_q}\rightarrow \infty$ as $\zeta_q$ goes to zero when $q\to \infty$. 

Since 
\[
\nu_0+\cdots +\nu_{K_{q}-1}+\nu_{K_{q}}=\frac{\nu_{K_q}}{\lambda_{K_q}} +\nu_{K_q},
\]
one also has
\[
\left( \frac{\nu_{K_q}}{\lambda_{K_q}} +\nu_{K_q}\right) b^{K_q}\geq (1-\zeta_q)\nu_{K_q}b^{K_q} + (1-\zeta_q)\sum_{k=0}^{{K_q}-1} \nu _k b^k
\]
by \eqref{ineq:lambda_nu_zeta}, or
\[
\left(\frac{1}{\lambda_{K_q}}+\zeta_q\right)\nu_{K_q}b^{K_q}\geq (1-\zeta_q)\sum_{k=0}^{{K_q}-1} \nu _k b^k,
\]
and thus 
\[
\frac{\nu_{K_q}b^{K_q}}{\sum_{k=0}^{{K_q}-1} \nu _k b^k}\geq \frac{1-\zeta_q}{\left(\frac{1}{\lambda_{K_q}}+\zeta_q\right)}
\]
for such sequences $\{ \zeta_{q}\}$ and $\{ K_q\}$. 
\end{proof}

\begin{proof}[Proof of Sublemma \ref{nu}]
By definition $b\ell'_n\leq \ell_n \leq \ell'_n$ for every $n\in \mathbb{N}$. 
Hence
\[
\sum_{n=1}^\infty \frac{1}{n^2}e^{ab(\ell_1'+\dots + \ell_n')}
\leq \sum_{n=1}^\infty \frac{1}{n^2}e^{a(\ell_1+\dots + \ell_n)}.
\]
Then by assumption \eqref{contra-shepp^a} it follows that
\[
\sum_{n=1}^\infty \frac{1}{n^2}e^{ab(\ell_1'+\dots + \ell_n')}<\infty.
\]
Therefore
\[
\frac{1}{(\nu_0+\dots +\nu_K)^2}e^{ab\sum_{k=0}^K \nu_k b^k}\to 0,
\]
as $K \rightarrow \infty$. 
Then there exists $C>0$ such that for all large $K\in \mathbb{N}$, we have
\[
\sum_{k=0}^K \nu_k b^k\leq C\ln (\nu_0+\dots + \nu_K).
\]
It follows from the assumption on $\lambda_K$ that 
\begin{equation} \label{ineq:log_nu}
    \nu_K b^K \leq \sum_{k=0}^K \nu_k b^k\leq C\ln (\nu_0+\dots + \nu_K)\leq C \ln (2\nu_{K}). 
\end{equation}
Note that in view of \eqref{contra-shepp^a} we have
\[
\sum_{n=1}^\infty \ell_n^2<\infty,
\]
by \cite{Fan-Karagulyan2021}*{Lemma 3.2}. 
Then we see that 
\[
\sum_{n=1}^\infty (b\ell'_n)^2=b^2\sum_{k=0}^\infty \nu_k b^{2k}\leq \sum_{n=1}^\infty \ell_n^2<\infty
\]
Hence $\nu_k b^{2k} \leq 1$, for all large values of $k$. 
It then follows from \eqref{ineq:log_nu} that
\[
\nu_K b^K \leq C \ln (2\nu_{K}) \leq \{2C\ln(2/b)\} K.
\]
Letting $C_0=2C\ln(2/b)>0$. 
Lemma is obtained. 
\end{proof}

\subsection{Proof of Theorem \ref{mainThm:ub_dvoretzky-hawkes}-(a)} \label{hawkes}

In this section, we give a necessary condition of Kahane-Hawkes type for non-uniform Dvoretzky covering of $\mathbb{T}$ (see \cites{Hawkes1973, Kahane1985}), and prove Theorem \ref{mainThm:ub_dvoretzky-hawkes}-(a). 
For a given sequence $\{\ell_n\}_{n \in \mathbb{N}}$, let
\begin{equation} \label{D}
    D=D(\{\ell_n\})=\limsup_{n \to \infty}\frac{\ell_1 + \dots + \ell_n}{\ln n}.
\end{equation}

\begin{lemma} \label{lem:shepp^a-hawkes}
Conditions \eqref{mD} and \eqref{cond:shepp^a} are equivalent. 
\end{lemma}
\begin{proof}
Assume we have \eqref{mD}. 
Then for every $\varepsilon>0$ there exists a sub-sequence $\{n_k\}_{k \in \mathbb{N}}$ so that
\[
\lim_{k \to \infty}\frac{(\ell_1 + \dots + \ell_{n_k})}{\ln n_k}>\frac{1-\varepsilon
}{m_f}.
\]
Fix an integer $r\geq 1$. 
We can assume that $2^r$ divides $n_k$, since for large $n_k$ the reminder, after dividing by $2^r$ will be negligible with respect to $n_k$. Since $\ell_n$ is decreasing then
\[
2^r\sum_{k=(2^r-1)m+1}^{2^rm} \ell_k \leq \sum_{k=1}^{2^rm} \ell_k.
\]
Hence
\[
\sum_{k=1}^{(2^r-1)m}\ell_k \geq \left(1-\frac{1}{2^r}\right)\sum_{k=1}^{2^rm}\ell_k.
\]
Then
\begin{align*}
    \sum_{n=(2^r-1)m+1}^{2^{r}m}\frac{1}{n^2}e^{a(\ell_1 + \dots + \ell_n)}
    &>e^{a \left( \ell_1 + \dots + \ell_{(2^r-1)m}\right)}\sum_{n=(2^r-1)m+1}^{2^{r}m}\frac{1}{n^2} \\
    &\geq e^{a\left(1-\frac{1}{2^r}\right) \left( \ell_1 + \dots + \ell_{2^rm}\right)}\frac{1}{2^{2r} m} \\
    &\geq e^{\frac{a (1-\varepsilon)}{ m_f } { \left(1-\frac{1}{2^r}\right)}\ln 2^r m }\frac{1}{2^{2r}m} =\frac{1}{2^r} {(2^rm)}^{\frac{a(1-\varepsilon)}{m_f }\left(1-\frac{1}{2^r}\right)-1} .
\end{align*}
The last term is unbounded since for every $a>m_f$ one can choose $r\geq 1$ and $\varepsilon>0$ such that $\frac{a(1-\varepsilon)}{m_f }\left( 1-\frac{1}{2^r}\right)-1>0$.

Next, we show by contraposition that \eqref{cond:shepp^a} implies \eqref{mD}. 
Assume 
\[
\limsup_{n \to \infty}\frac{\ell_1 + \dots + \ell_n}{\ln n}<\frac{1}{m_f}.
\]
Then there is $\varepsilon>0$ so that for all sufficiently large $n \geq 1$
\[
\frac{\ell_1 + \dots + \ell_n}{\ln n}\leq \frac{1-\varepsilon}{m_f}.
\]
Let $a>m_f$. 
Then
\[
a(\ell_1 + \dots + \ell_n) \leq \frac{a(1-\varepsilon)\ln n}{m_f}= \rho \ln n
\]
where we let $\rho =a(1-\varepsilon)/m_f$.
Note that if $a$ is sufficiently close to $m_f$, then $\rho<1$.
Hence
\[
\sum_{n=1}^\infty \frac{1}{n^2}e^{a(\ell_1+\dots +\ell_n)}<\sum_{n=1}^\infty \frac{1}{n^2}e^{\rho\ln n}=\sum_{n=1}^\infty \frac{1}{n^{2-\rho}}<\infty.
\]
This contradicts \eqref{cond:shepp^a}.
\end{proof}

\begin{proof}[Proof of Theorem \ref{mainThm:ub_dvoretzky-hawkes}-(a)]
The first half of Theorem \ref{mainThm:ub_dvoretzky-hawkes}-(a) readily follows from Proposition \ref{mainProp:DtoS} and Lemma \ref{lem:shepp^a-hawkes}. 
For the second half of Theorem \ref{mainThm:ub_dvoretzky-hawkes}-(a) note that if the inequality is strict, then by Lemma \ref{lem:shepp^a-hawkes} we have that 
\[
\sum_{n=1}^\infty \frac{1}{n^2}e^{m_f(\ell_1+ \dots + \ell_n)}=\infty.
\]
follows from \cite{Fan-Karagulyan2021}*{Proposition 5.3}. 
\end{proof}



\subsection{Proof of Theorem \ref{mainThm:ub_dvoretzky-hawkes}-(b)}

Let $f\in L^1(\mathbb{T})$ be a probability density function so that $m_f>0$ and $\overline{\dim}_\mathrm{B}K_f<1$, and let $\{\ell_n\}_{n \in \mathbb{N}}$ be a sequence satisfying \eqref{mD}. 
It is enough to show the sufficiency of \eqref{mD} for $\mu_f$-Dvoretzky covering for $\mathbb{T}$. 
Set $D=D(\{\ell_n\}_{n\in \mathbb{N}})$ as in \eqref{D}. 

By Lemma \ref{lem:shepp^a-hawkes}, we have \eqref{cond:shepp^a}. 
Once we have \eqref{cond:kahane^m}, it follows from Proposition \ref{cond:fk} that the circle is $\mu_f$-Dvoretzky covered for $\{\ell_n\}_{n\in \mathbb{N}}$. 
Henceforth, we show the condition \eqref{cond:kahane^m}, that is, $\mathrm{Cap}_{m_f}(K_f)=0$.

Let $\widetilde{K_f}=m_f K_f=\{ m_fx\colon x\in K_f\}$. 
Given a Borel probability measure $\sigma$ supported on $K_f$, let $\tilde{\sigma} =\sigma \circ (m_f)^{-1}$, which defines a Borel probability measure supported on the scaled set $\widetilde{K_f}$. 
Then we see 
\begin{equation*} 
    \int_{K_f} \int_{K_f} \exp \left\{ m_f\sum_{k=1}^\infty (\ell_k - |t-s|)_+\right\} \, d\sigma(s)d\sigma(t) =\infty
\end{equation*}
is equivalent to
\begin{equation*} 
    \int_{\widetilde{K_f}} \int_{\widetilde{K_f}} \exp \left\{ \sum_{k=1}^\infty (m_f\ell_k - |u-v|)_+\right\} \, d\tilde{\sigma}(u)d\tilde{\sigma}(v)=\infty.
\end{equation*}
Namely, to show $\mathrm{Cap}_{m_f}(K_f)=0$ is equivalent to show $\mathrm{Cap} (\widetilde{K_f})=0$, that is, the Kahane capacity of  $\widetilde{K_f}$ with respect to the sequence $\{m_f\ell_n\}_{n \in \mathbb{N}}$ is zero. 

Now, since the upper box-counting dimension is invariant under scaling, one has $\overline{\dim} _\mathrm{B}\widetilde{K_f} =\overline{\dim} _\mathrm{B}K_f<1\leq m_f D=D(m_f\ell_n)$. 
Then, in view of \cite{Kahane1985}*{Theorem 4 (1) (page 160)}, there is a Dvoretzky covering (in the classical sense) of the set $\widetilde{K_f}$ by the sequence $\{m_f\ell_n\}_{n\in \mathbb{N}}$, and thus $\mathrm{Cap}(\widetilde{K_f})=0$ follows by the Kahane criterion.

\section{Proof of Theorem \ref{mainThm:perturbation}}
\addtocontents{toc}{\protect\setcounter{tocdepth}{1}}

Let $g$ be a density function. Recall that $m_g=\essinf_\mathbb{T} g$ and $K_g\subset \mathbb{T}$ is the set of essential infimum points of $g$. 

\subsection{Auxiliary covering} \label{aux_covering}

We begin with a generic inequality. 
Given a decreasing sequence of positive numbers with $\ell_n\in (0,1)$ and $N\in \mathbb{N}$, let $I_n=(\omega_n-(\ell_n/2),\omega_n+(\ell_n/2))$ and 
\[
U_N=\bigcup _{j=1}^N I_j. 
\]
Let $A$ be a compact subset of $\mathbb{T}$. 
Given $\varepsilon>0,$ we consider a minimal system of balls of radius $\varepsilon$ covering $A,$ say $B\left(x_{1}, \varepsilon\right), \ldots, B\left(x_{M}, \varepsilon\right)$ then $M=M(A,\varepsilon)\in \mathbb{N}$ is the $\varepsilon$-covering number of $A$. 
Since
\[
A \subset \bigcup_{q=1}^{M} B(x_{q}, \varepsilon)
\]
we have
\[
\Pr\left(A \not\subset U_{N}\right) \leq \sum_{q=1}^{M} \Pr\left(B\left(x_{q}, \varepsilon\right) \not\subset U_{N}\right).
\]
For every $N\in \mathbb{N}$, we have
\[
\Pr\left(B(x_q, \varepsilon) \not\subset U_{N}\right)=\prod_{j=1}^{N} \Pr\left(B(x_q, \varepsilon) \not\subset  I_{j}\right)=\prod_{j=1}^{N} \Pr \left(\omega_j \notin I_{j,q}^{\varepsilon}\right),
\]
where
\begin{equation} \label{interval}
    I_{j,q}^{\varepsilon}=\{\omega_j\in \mathbb{T} \colon B(x_q, \varepsilon) \subset I_j\}=B\left(x_q, \frac{\ell_j}{2}-\varepsilon\right).
\end{equation}
Suppose that we are given a probability density function $g\in L^1(\mathbb{T})$. 
Recall that $\mu _{g}$ denotes the distribution of $\omega _n$ so that $d\mu _{g}=g\, dt$. 
Then one has
\[
\Pr (\omega_j \not\in I_{j,q}^{\varepsilon})=1-\mu_{g}\left(I_{j,q}^{\varepsilon}\right)=1-\int_{I_{j,q}^{\varepsilon}}g(t)\, dt,
\]
and hence  
\begin{align*}
    \Pr\left(B\left(x_{q}, \varepsilon\right) \not\subset U_{N}\right) =\prod_{j=1}^{N}\Pr (\omega_j \not\in I_{j,q}^{\varepsilon}) 
    =\prod_{j=1}^{N}\left( 1-\mu_{g}\left(I_{j,q}^{\varepsilon}\right)\right) &\leq \exp\left( -\sum_{j=1}^{N}\mu_{g}\left(I_{j,q}^{\varepsilon}\right) \right) \\ 
    &= \exp \left(-\sum_{j=1}^{N} \int_{I_{j,q}^{\varepsilon}}g(t)\,dt\right)
\end{align*}
for each $q\in \{ 1,\dots ,M\}$. 
Therefore one has
\begin{equation} \label{ineq:A1}
\Pr\left(A \not\subset U_{N}\right) \leq \sum_{q=1}^M \exp \left(-\sum_{j=1}^{N} \int_{I_{j,q }^{\varepsilon}}g(t)\,dt\right) 
\end{equation}
for every $N\in \mathbb{N}$. 

Inequality \eqref{ineq:A1} will be key in obtaining $\mu_{g}$-Dvoretzky covered for the set $K_g$ for some $g$, to be constructed below.
Note that the assumption \eqref{mD} implies that for any $a>m_f$ 
\begin{equation}\label{condition:ln}
    \limsup _{n\to \infty}\frac{1}{n} \exp \{ a\left(\ell_{1}+\cdots+\ell_{n}\right) \} =\infty.
\end{equation}
Then there is an infinite set $\Lambda \subset \mathbb{N}$ such that for every $n\in \Lambda$ one has $\ell _n>1/(2n)$. 
Indeed, if we let $u_n=\frac{1}{n} \exp \left(\ell_{1}+\cdots+\ell_{n}\right) $ and $\Lambda =\{ n\in \mathbb{N} \colon u_n\geq \sup _{m<n}u_m\}$, then it follows from \eqref{condition:ln} with $a=1$ (note that $m_f<1$) that $\#\Lambda =\infty $ and $\lim _{n\in \Lambda}u_n=\infty$, and hence for every $n\in \Lambda$ one has  $\ell_n\geq \ln \frac{n}{n-1} >\frac{1}{2n}$ since $u_n\geq u_{n-1}$. 
We now state a key lemma:
\begin{lemma} \label{lem:main}
    Let $\rho \in [0,1]$, $\{\ell_n\}_{n\in \mathbb{N}}$ be a decreasing sequence with $\sum_{n=1}^\infty\ell_n=\infty$, and $A\subset \mathbb{T}$ a non-empty compact set with $|A|=0$. 
    Then there exist a probability density function $g\in L^1(\mathbb{T})$ with $m_g=\rho$ and $K_g=A$, a sequence of blocks of consecutive integers $\left\{ \left[n_1^{(k)}, n_{2}^{(k)}\right]\cap \mathbb{N}\right\}_{k\in \mathbb{N}}$ with $n_{2}^{(k)}\in \Lambda$, and a sequence of collections of balls $\left\{\mathcal{B}^{(k)}=\left\{B_q^{(k)}\right\}_{q=1}^{M_k}\right\}_{k\in \mathbb{N}}$ with radii $\{\varepsilon_k\}_{k\in \mathbb{N}}$ and  $\varepsilon_k \in \left[\ell_{n_2^{(k)}}/4, \ell_{n_2^{(k)}}/2\right)$ such that
    \begin{enumerate}
    \item[1)] For every $k\in \mathbb{N}$ and some $C>0$
    \[K_f \subset \cup \mathcal{B}^{(k)}=\cup_{q=1}^{M_k} B_q^{(k)}\]
    and $M_k=\# \mathcal{B}^{(k)}\leq C\varepsilon_k^{-\overline{\dim}_\mathrm{B} A}$. 
    \item[2)] For every $k\in \mathbb{N}$, 
    \[
    \frac{\ell_{1} + \dots + \ell_{n_{1}^{(k)}}}{\ell_{n_{1}^{(k)}} + \dots + \ell_{n_{2}^{(k)}}}<\frac{1}{2^k}.
    \]
    \item[3)] 
    For every $k\in \mathbb{N}$ and every $j \in \left[n_1^{(k)}, n_{2}^{(k)}\right]\cap \mathbb{N}$, we have that \[\int_{I_{j,q}^{\varepsilon_k}}g(t)\,dt\geq |I_{j,q}^{\varepsilon_k}|,\] 
    where $I_{j,q}^{\varepsilon_k}\subset B_q^{(k)}$ is the ball of the from \eqref{interval} around $x_q$, where  $x_q$ is the center of the ball $B_q^{(k)}\in \mathcal{B}^{(k)}$. 
    \end{enumerate}
\end{lemma}

We will postpone the proof of this lemma to Section \ref{proof:mainlema}, and proceed the argument for the proof of Theorem \ref{mainThm:perturbation}. 
We use inequality \eqref{ineq:A1} with Lemma \ref{lem:main} to show that $K_{g}$ is $\mu_{g}$-Dvoretzky covered.  
For each family $\mathcal{B}^{(k)}$ of balls of radius $\varepsilon_k$, it follows from Lemma \ref{lem:main}-(2,4) that 
\begin{align*}
    \sum_{j=1}^{N} \int_{I_{j,q}^{\varepsilon_k}}g(t)\,dt
    &=\left(\sum_{j=n_1^{(1)}}^{n_2^{(1)}}+\sum_{j=n_1^{(2)}}^{n_2^{(2)}}+\dots +\sum_{j=n_1^{(k)}}^{n_2^{(k)}}+\sum_{\substack{j=1,\\j\notin D}}^{N}\right)\int_{I_{j,q}^{\varepsilon_k}}g(t)\,dt \\
    &\geq \sum_{j=n_1^{(k)}}^{n_2^{(k)}}\int_{I_{j,q}^{\varepsilon_k}}g(t)\,dt \geq \sum_{j=n_1^{(k)}}^{n_2^{(k)}}(\ell_j-2\varepsilon_k) \geq \frac{1}{1+2^{-k}}\sum_{j=1}^{n_2^{(k)}}\ell_j-\sum_{j=1}^{n_2^{(k)}}2\varepsilon_k 
\end{align*}
for every $N\in \mathbb{N}$, where $D=\cup _{i=1}^k([n_1^{(i)}, n_{2}^{(i)}]\cap \mathbb{N})$. 
Hence by \eqref{ineq:A1} one has
\begin{align} \label{ineq:A2}
    \Pr\left(A \not\subset U_{N}\right) 
    &\leq \sum_{q=1}^{M_k} \exp \left(-\sum_{j=1}^{N} \int_{I_{j,q}^{\varepsilon_k}}g(t)\,dt\right) \nonumber \\
    &\leq \sum_{q=1}^{M_k} \exp \left\{-\left(\frac{1}{1+2^{-k}}\sum_{j=1}^{n_2^{(k)}}\ell_j-\sum_{j=1}^{n_2^{(k)}}2\varepsilon_k\right)\right\} \nonumber \\
    &=\exp \left\{-\left(\frac{1}{1+2^{-k}}\sum_{j=1}^{n_2^{(k)}}\ell_j-\sum_{j=1}^{n_2^{(k)}}2\varepsilon_k\right)+\ln M_k\right\}
\end{align}
for every $N\in \mathbb{N}$ with $n_2^{(k)}\in \Lambda$. 
Taking $\varepsilon_k=1/(2n_2^{(k)})$, one has
\[
-\frac{1}{1+2^{-k}} \sum_{j=1}^{n_2^{(k)}} \ell_j + \sum_{j=1}^{n_2^{(k)}} \frac{1}{n_2^{(k)}}+\ln M_k\leq -\frac{1}{1+2^{-k}} \sum_{j=1}^{n_2^{(k)}} \ell_j +1 +\ln(2n_2^{(k)})
\]
as $\overline{\dim}_\mathrm{B}A\leq 1$. 
Since the right-hand side is less than zero by \eqref{mD}, it follows that $\Pr (A\not\subset U_N)\to 0$ as $N\to \infty$, that is $K_{g}(=A)$ is $\mu_{g}$-Dvoretzky covered. 

\subsection{Concluding the proof of Theorem \ref{mainThm:perturbation}} 

Henceforth, we will define a function $f_0$ having the desired properties as stated.
We choose a set $B\subset \mathbb{T}\setminus K_f$ so that $|B|<\varepsilon/2$ and $f(x)\geq 1$ for all $x \in B$ (recall that $\|f\|_{L^1(\mathbb{T})}=1$). 
Since $|K_f|=0$, we then take a covering $\mathcal{U}=\{ U_i\} _{i=1}^k$ of $K_f$ such that $|\cup _{i=1}^k U_i|/|B|<1-m_f$. 
Note that one can assume $B\cap \cup _{i=1}^k U_i=\emptyset $. 
By Lemma \ref{lem:main} with $\rho =m_f$ and $A=K_f\cap U_i$, one obtains a function $g_i$ so that $m_{g_i}=m_{f}$ and  $K_{g_i}=K_f\cap U_i$ for $i=\{ 1,\dots,k\}$. 
Let $U=\cup _{i=1}^k U_i$. 
Define  
\[
g_{U}=g_1|_{U_1}+\cdots +g_k|_{U_k},
\]
and set
\[
g=g_U + f|_{{\mathbb{T}\setminus U}}. 
\]
Then one has $m_{g}=m_f$, $K_{g}=K_f$, and 
\[
\|g\|_{L^1(\mathbb{T})}=\|f\|_{L^1(\mathbb{T})}-\int_U f\, dx + \int_U g_U\, dx=1-\int_U f\, dx + \int_U g_U\, dx. 
\]
Let $c=-\int_U f\, dx + \int_U g_U\, dx$, and define $f_0 =g - \frac{c}{|B|}\mathbbm{1}_B$. 
Then we have 
\[
|\{ x\in \mathbb{T}\colon f(x)\neq f_0(x)\}|\leq |U|+|B|<2|B|<\varepsilon. 
\]
Next, we will show $m_{f_0}=m_f$ and $K_{f_0}=K_f$. 
If $c\leq 0$, then it is clear from the definition of $f_0$ that $f_0|_B\geq g|_B$ and $f_0|_U=g|_U=g_U$. 
Suppose $c>0$. 
Note that $\int_U f\, dx<\int_U g_U\, dx\leq |U|$, and hence $c \leq |U|$. 
It follows that for $x\in B$ we have
\[
(f_0|_B)(x) \geq f(x)-\frac{c}{|B|}\geq 1-\frac{|U|}{|B|}>m_f.
\]
This implies $m_{f_0}=m_f$ and $K_{f_0}=K_f$ in this case as well. 

We also have 
\[
\|f_0\|_{L^1(\mathbb{T})}=\|g\|_{L^1(\mathbb{T})}-c=1+c-c=1.
\]
It remains to show that the circle $\mathbb{T}$ is $\mu _{f_0}$-Dvoretzky covered by a sequence satisfying \eqref{mD}. 
Note first that, by \eqref{mD} and Lemma \ref{lem:shepp^a-hawkes}, one has 
\[
\sum_{n=1}^{\infty} \frac{1}{n^{2}} \exp \{ a\left(\ell_{1}+\cdots+\ell_{n}\right)\} =\infty 
\]
for every $a>m_{f_0}(=m_f)$. 
Thus by \cite{Fan-Karagulyan2021}*{Proposition 5.1} the set $\mathbb{T}\setminus K_{f_0}$ is $\mu_{f_0}$-Dvoretzky covered. 
Next, note that $f_0|_U=g|_U=g_1|_{U_1}+\cdots +g_k|_{U_k}$. 
Here, note also that $U=\cup _{i=1}^k U_i$ and $K_{g_i}=K_f\cap U_i=K_{f_0}\cap U_i$ is $\mu_{g_i}$-Dvoretzky covered for each $i=\{ 1,\dots,k\}$ by the argument in Section \ref{aux_covering}. 
Hence $K_{f_0}$ is $\mu_{f_0}$-Dvoretzky covered.
Theorem \ref{mainThm:perturbation} is obtained.

\subsubsection{Proof of lemma \ref{lem:main}} \label{proof:mainlema}

For each $k\in \mathbb{N}$, let 
\[
\mathcal{D}^{(k)}=\left\{ \left[ \frac{i}{2^k},\frac{i+1}{2^k}\right) \subset \mathbb{T}\colon i=0,\dots ,2^k-1\right\}
\]
and $\mathcal{D}^{(k)}_A=\{ J\in \mathcal{D}^{(k)}\colon J\cap A\neq \emptyset \}$. 

For each $J\in \mathcal{D}^{(k)}_A$, choose an arbitrary point $y^{(k)}_J\in J\setminus A$, and let $Y^{(k)}$ be the collection of all such points, i.e.,
\[
Y^{(k)}=\left\{y_J^{(k)}\in \mathbb{T}\colon J\in \mathcal{D}^{(k)}_A\right\}.
\]
Sometimes, we may denote $y\in Y^{(k)}$ by $y^{(k)}$ to specify its generation. 
Note that $\#Y^{(k)}\leq 2^k$. 
Let 
\[
\delta_k=d\left( Y^{(k)},A\right). 
\]
Note that $\delta_k>0$ for all $k \in \mathbb{N}$. 

For each $k\in \mathbb{N}$, we will define sequences of integers $n_1^{(k)}, n_2^{(k)}\in \mathbb{N}$, positive small numbers $\eta_k>0$, $\varepsilon_k>0$, and finite families $\mathcal{Y}^{(k)}$  and $\mathcal{B}^{(k)}$ of balls of radii $\eta_k$ and $\varepsilon_k$, respectively.   
First, we take $n_1^{(1)}\in \mathbb{N}$ so large that $\ell_{n_1^{(1)}}\leq \delta_1/4$.
We next choose $n_2^{(1)}\in \Lambda$ and blocks $(\ell_{n_1^{(1)}}, \dots, \ell_{n_{2^{(1)}}})$ so that 
\begin{equation} \label{ratio-1}
    \frac{\ell_{1} + \dots + \ell_{n_{1}^{(1)}}}{\ell_{n_{1}^{(1)}} + \dots + \ell_{n_{2}^{(1)}}}<\frac{1}{2}.
\end{equation}
Note that such a choice is possible since $\sum_{n=1}^\infty \ell_n = \infty$.
Take $\eta_1\in (0,\delta_1/4]$, and let $\mathcal{Y}^{(1)}=\{ B(y,\eta_1)\colon y\in Y^{(1)}\}$. 
Note that for every ball $B$ of $\mathrm{diam} B\leq \delta_1/4$ which intersects $A$, we have 
\begin{equation} \label{disjoint-1}
    B(y,\eta_1) \cap B = \emptyset
\end{equation} 
for every $B(y,\eta_1)\in \mathcal{Y}^{(1)}$. 
To conclude the first step of our construction, we take $\varepsilon_1 \in \left[\ell_{n_2^{(1)}}/4,\ell_{n_2^{(1)}}/2\right)$ to be the dyadic number so that $-\log_2\varepsilon_1\in \mathbb{N}$, and set $\mathcal{B}^{(1)}=\mathcal{D}^{(-\log_2\varepsilon_1)}_A$. 
Hence $A\subset \sqcup \mathcal{B}^{(1)}$. 
Note that every $B\in \mathcal{B}^{(1)}$ fulfills \eqref{disjoint-1} since 
\[
\mathrm{diam}B=2\varepsilon_1<\ell_{n_2^{(1)}}<\ell_{n_1^{(1)}}\leq \delta_1/4. 
\]
One obtains a collection $\delta_1, n_1^{(1)}, n_2^{(1)}, \eta_1, \varepsilon_1,\mathcal{Y}^{(1)},  \mathcal{B}^{(1)}$. 

Next, we take $n_1^{(2)}\in \mathbb{N}$ so large that $\ell_{n_1^{(2)}}\leq \delta_2/4$. 
Here one can assume $\delta_2\leq \delta_1/2$. 
Then choose $n_2^{(2)}\in \Lambda$ and blocks $\left( \ell_{n_1^{(2)}}, \dots, \ell_{n_{2}^{(2)}}\right)$ so that 
\begin{equation} \label{ratio-2}
    \frac{\ell_{1} + \dots + \ell_{n_{1}^{(2)}}}{\ell_{n_{1}^{(2)}} + \dots + \ell_{n_{2}^{(2)}}}<\frac{1}{2^2}.
\end{equation}
Take $\eta_2\in (0,\delta_2/4]$, and define $\mathcal{Y}^{(2)}=\{B(y,\eta_2)\colon y\in Y^{(2)}\}$.
It follows that for every ball $B$ of $\mathrm{diam} B\leq \delta_2/4$ which intersects $A$, we have 
\begin{equation} \label{disjoint-2}
    B(y^{(i)},\eta_i) \cap B = \emptyset
\end{equation} 
for every $B(y^{(i)},\eta_i)\in \mathcal{Y}^{(i)}$, $i\in \{1,2\}$.
Take $\varepsilon_2 \in \left[\ell_{n_2^{(2)}}/4,\ell_{n_2^{(2)}}/2\right)$ to be the dyadic number so that $-\log_2\varepsilon_2\in \mathbb{N}$, and set $\mathcal{B}^{(2)}=\mathcal{D}^{(-\log_2\varepsilon_2)}_A$. 
Hence $A\subset \sqcup \mathcal{B}^{(2)}$. 
Note that every $B\in \mathcal{B}^{(2)}$ fulfills \eqref{disjoint-2} since
\[
\mathrm{diam}B=2\varepsilon_2<\ell_{n_2^{(2)}}<\ell_{n_1^{(2)}}\leq \delta_2/4. 
\]
Note also that every $B\in \mathcal{B}^{(2)}$ fulfills \eqref{disjoint-1} as well since 
\[
\mathrm{diam}B<\delta_2/4\leq \delta_1/8. 
\]
Hence we obtain a collection $\delta_2, n_1^{(2)}, n_{2}^{(2)}, \eta_2, \varepsilon_2, \mathcal{Y}^{(2)}, \mathcal{B}^{(2)}$. 

Repeating the procedure above, we will define $\delta_k, n_1^{(k)}, n_{2}^{(k)}, \eta_k, \varepsilon_{k}, \mathcal{Y}^{(k)}, \mathcal{B}^{(k)}$ for every $k\in \mathbb{N}$. 
More specifically, on the $k$-th step, take $n_1^{(k)}\in \mathbb{N}$ so large that $\ell_{n_1^{(k)}}\leq \delta_k/4$. 
As earlier, we can assume $\delta_k\leq \delta_{k-1}/2$. 
Then choose $n_2^{(k)}\in \Lambda$ and blocks $\left( \ell_{n_1^{(k)}}, \dots, \ell_{n_{2}^{(k)}}\right)$ so that 
\begin{equation} \label{ratio-k}
    \frac{\ell_{1} + \dots + \ell_{n_{1}^{(k)}}}{\ell_{n_{1}^{(k)}} + \dots + \ell_{n_{2}^{(k)}}}<\frac{1}{2^k}.
\end{equation}
Take $\varepsilon_k \in \left[\ell_{n_2^{(k)}}/4,\ell_{n_2^{(k)}}/2\right)$ to be the dyadic number so that $-\log_2\varepsilon_k\in \mathbb{N}$, and set a covering of $K$ by $\mathcal{B}^{(k)}=\mathcal{D}^{(-\log_2\varepsilon_k)}_K$. 
Let $\eta_k\in (0,\delta_k/4]$, and $\mathcal{Y}^{(k)}=\{B(y,\eta_k)\colon y\in Y^{(k)}\}$.
Then we see that every ball $B$ of $\mathrm{diam} B\leq \delta_k/4$ which intersects $A$, we have 
\begin{equation} \label{disjoint-k}
    B(y^{(i)},\eta_i) \cap B = \emptyset
\end{equation} 
for every $B(y^{(i)},\eta_i)\in \mathcal{Y}^{(i)}$, $i\in \{1,\dots,k\}$.
Every ball $B\in \mathcal{B}^{(k)}$ fulfills \eqref{disjoint-k} as well. 

Let $M(A,\varepsilon_k)$ be the minimum covering number for $A$ by balls of radius  $\varepsilon_k$. 
Then one has  $\#\mathcal{B}_k/2\leq M(A,\varepsilon_k)\leq \#\mathcal{B}_k$. 
Hence $\#\mathcal{B}_k\approx M(A,\varepsilon_k)\lesssim  \varepsilon_k^{-\overline{\dim }_\mathrm{B}A}$ as $k\to \infty$. 

Now, we define a function $g\in L^1_+(\mathbb{T})$.
For each $k\in \mathbb{N}$, let $\rho_k=\rho +(1/k)$. 
For $x\in \cup_{k=1}^{\infty}\cup \mathcal{Y}^{(k)}$, define
\[
g(x)=\begin{cases} \rho_{k(x)}, & \#\{k\colon x\in \cup \mathcal{Y}^{(k)} \}<\infty, \\ \rho ,& \text{otherwise},\end{cases}
\]
where $k(x)=\max \{k\in \mathbb{N} \colon x\in \cup \mathcal{Y}^{(k)}\}$. 
We set the value of $g$ on $\mathbb{T}\setminus (\cup _{n\in \mathbb{N}}\cup \mathcal{Y}^{(n)})$ to be the constant, say $\gamma >0$, so that $\|g\|_{L^1}=1$. 
Note that $\gamma \geq 1\geq m_g$. 

It remains to show item (3). 
Let $B(x_q,\varepsilon_k)\in \mathcal{B}^{(k)}$, and consider $I_{j,q}^{\varepsilon_k}$ for $j\in \{n_1^{(k)},\dots ,n_2^{(k)}\}$. 
Recall here that $I_{j,q}^{\varepsilon_k}=B(x_q,(\ell_j/2)-\varepsilon_k)$. 
Note that any ball $B(y^{(i)},\eta_i)\in \mathcal{Y}^{(i)}$ for $i\leq k$ never intersects $I_{j,q}^{\varepsilon_k}$ since
\[
\mathrm{diam}I_{j,q}^{\varepsilon_k}=\ell_j-2\varepsilon_k\leq \ell_{n_1^{(k)}}-2\varepsilon_k<\ell_{n_1^{(k)}}\leq \delta_k/4.
\]
Hence
\begin{align*} 
    \int_{I_{j,q}^{\varepsilon_k}}g\, dt
    &=\int_{I_{j,q}^{\varepsilon_k}\cap \left(\cup_{i>k}^{\infty}\cup\mathcal{Y}^{(i)}\right)}g\, dt + \int_{I_{j,q}^{\varepsilon_k}\setminus\left( \cup_{i>k}^{\infty}\cup\mathcal{Y}^{(i)}\right)}g\, dt \\
    &\geq \int_{I_{j,q}^{\varepsilon_k}\setminus \left(\cup_{i>k}^{\infty}\cup\mathcal{Y}^{(i)}\right)}g\, dt \\
    &\geq \gamma \left| I_{j,q}^{\varepsilon_k}\setminus \left(\cup_{i>k}^{\infty}\cup\mathcal{Y}^{(i)}\right)\right| \\
    &\geq \gamma \left(\left| I_{j,q}^{\varepsilon_k}\right| -\sum_{i>k}^{\infty}\sum_{B(y,\eta_i)\in \mathcal{Y}^{(i)}}|B(y,\eta_i)|\right) \\
    &\geq \gamma \left(\left| I_{j,q}^{\varepsilon_k}\right| -\sum_{i>n}^{\infty}2^i2\eta_i\right) =\gamma \left| I_{j,q}^{\varepsilon_k}\right| \left(1-\frac{\sum_{i>k}^{\infty}2^{i+1}\eta_i}{|I_{j,q}^{\varepsilon_k}|}\right).
\end{align*}
Given $\zeta \in (0,1)$, one can take $\{\eta_k\}_k$ so that $(\sum_{i>k}^{\infty}2^{i+1}\eta_i)/|I_{j,q}^{\varepsilon_k}|<\zeta $, and thus item (3) is obtained.

\section{Proof of Theorem \ref{mainThm:h_d_fail1}}

The idea is to keep the lower box-counting dimension of $K_f$ larger than $(\ell_1 + \dots +\ell_n)/\ln n$. Because of this $K_f$ will always be too large to be covered by $\{\ell_n\}_{n\in \mathbb{N}}$. When the expression $(\ell_1 + \dots +\ell_n)/\ln n$ approaches $1/m_f$ and the sequence $\{\ell_k\}_{k=1}^n$ starts exhibiting strong covering properties, we increased the lower box-counting dimension of $K_f$ at scales close to $\ell_n$ so that it will be even closer to $1/m_f$. Thus, the sequence will not be strong enough to cover $K_f$ at these scales.  After this, we drop the value of $(\ell_1 + \dots +\ell_n)/\ln n$ by creating long intervals without any $\ell_n$'s in them. Then, we also drop the lower box-counting dimension of $K_f$ in these scales, so that the Hausdorff dimension of $K_f$ will eventually be less than $1$.

We now state a lemma:

\begin{lemma} \label{lem:inhomocantor}
Let $\{ \varepsilon_k\} _{k\in \mathbb{N}}$ be a sequence so that $\varepsilon_k\searrow 0$ as $k\to \infty$. 
Then there exists a compact set $A\subset \mathbb{T}$ with $\dim _\mathrm{H}A<\overline{\dim} _\mathrm{B}A=1$, a Borel probability measure $\sigma _0$ supported on $A$, and a sequence of disjoint intervals $\{ [u_k,v_k]\} _{k\in \mathbb{N}}$ with $u_k>0$, $v_{k+1}<u_k$, and $u_k/v_k<c_0$ for some constant $c_0>0$ such that for all $t \in A$ and $r \in [u_k,v_k]$,  $\sigma_0(B(t,r))\leq c_1 r^{1-(\varepsilon_k/4)}$ for some constant $c_1>0$. 
\end{lemma}

We postpone the proof of this lemma until the end of this section and proceed the argument. 
Let $f\in L^1(\mathbb{T})$ be a probability density function for which $m_f>0$, $K_f=A$ and $f$ is Lipschitz on $A$. 
Then by \cite{Fan-Karagulyan2021}*{Corollary 1.1-(2)} the conditions \eqref{cond:shepp^a} and \eqref{cond:kahane^m} will both be necessary and sufficient for $\mu_f$-Dvoretzky covering. 
However, we will show that $\mathrm{Cap}_{m_f}(K_f)>0$, for the sequence $\ell_n$ to be defined below. 
Hence there can be no $\mu_f$-Dvoretzky covering. 

Let $h(x)=\frac{1}{m_f} \frac{\ln x}{x}$. 
Choose an increasing sequence $\{n_k\}_{k\in \mathbb{N}\cup \{0\}}$ recursively as follows. 
Let $n_0=0$. 
Take $n_1$ to be a positive integer so that $h(n_1)\in [u_{j_1}/(1-\varepsilon _{j_1}),v_{j_1}/(1-\varepsilon _{j_1})]$ for some (large) $j_1\in \mathbb{N}$. 
Note that this is possible since $h(n)/h(n+1)\rightarrow 1$, as $n \rightarrow \infty$ and $u_n/v_n<c_0$, so for large values of $n$, $h(n)$ cannot miss the intervals $\{[u_k,v_k]\}_{k \geq 1}$.
Next, take $n_2>n_1$ so that $h(n_2)\in [u_{j_2}/(1-\varepsilon _{j_2}),v_{j_2}/(1-\varepsilon _{j_2})]$ for some $j_2\in \mathbb{N}$. 
Note that $j_2>j_1$. 
Once $n_k\in \mathbb{N}$ is chosen, take $n_{k+1}\in \mathbb{N}$ with $n_{k+1}>n_k$ so that $h(n_{k+1})\in [u_{j_{k+1}}/(1-\varepsilon _{j_{k+1}}),v_{j_{k+1}}/(1-\varepsilon _{j_{k+1}})]$ for some $j_{k+1}\in \mathbb{N}$. 
Note that $j_{k+1}>j_k$. 
During this procedure, one can choose $\{n_k\}$ to grow fast enough so that 
\[
n_1^{1-\varepsilon_{j_1}}\cdots n_k^{1-\varepsilon_{j_k}} \leq n_k^{1-(\varepsilon_{j_k}/2)}. 
\]

For $k\in \mathbb{N}$, define
\[
b_k=\left(\frac{1-\varepsilon_{j_{k}}}{m_f}\right) \frac{\ln n_{k}}{n_{k}}, 
\]
and for $n\in \mathbb{N}$ with $n_0+n_1+\cdots +n_k<n\leq n_0+n_1+\cdots +n_k+n_{k+1}$ for $k\in \mathbb{N}\cup \{0\}$, define 
\[
\ell _n=b_{k+1} =\left(\frac{1-\varepsilon_{j_{k+1}}}{m_f}\right) \frac{\ln n_{k+1}}{n_{k+1}}.
\]
Note that if $n_k$ grows fast enough, then
$$
\frac{\sum_{k=1}^{n_1+\cdots +n_q} \ell_k}{\ln (n_1 + \dots  + n_q)}= \frac{1}{m_f}\frac{\sum_{i=1}^{q} (1-\varepsilon_{j_i}) \ln n_{i}}{\ln (n_1 + \dots  + n_q)}\geq \frac{1}{m_f}(1-2\varepsilon_{j_q}).
$$
This implies \eqref{mD}. Note that for $(s,t)\in K_f\times K_f$ with $|t-s|\in (b_{q+1}, b_q]$, one has 
\begin{align*}
    \sum_{k=1}^\infty (\ell_k - |t-s|)_+=\sum_{k=1}^{n_1+\cdots +n_q} (\ell_k - |t-s|) 
    \leq \sum_{k=1}^{n_1+\cdots +n_q} \ell_k 
    &=\frac{1}{m_f} \sum_{i=1}^{q} (1-\varepsilon_{j_i})n_ib_i \\
    &=\frac{1}{m_f} \sum_{i=1}^{q} (1-\varepsilon_{j_i}) \ln n_{i} \\
    &\leq \frac{1}{m_f} (1-(\varepsilon_{j_q}/2)) \ln n_{q} . 
\end{align*}
Next, for $t\in K_f$ and $q\in \mathbb{N}$, let $A_t(q)=(t+b_{q+1},t+b_{q}]\cup [t-b_{q},t-b_{q+1})$. 
Note that $A_t(q)\subset B(t,q)$. 
It follows that for each $t\in K_f$ 
\begin{align*}
    \int_{K_f} \exp \left\{ m_f \sum_{k=1}^\infty (\ell_k - |t-s|)_+\right\} \, d\sigma_0(s)
    &=\sum_{q=1}^\infty\int_{K_f\cap A_t(q)} \exp \left\{ m_f \sum_{k=1}^\infty (\ell_k - |t-s|)_+\right\} \, d\sigma_0(s) \\
    &\leq \sum_{q=1}^\infty\int_{K_f\cap A_t(q)} \exp \left\{ (1-(\varepsilon_{j_q}/2)) \ln n_{q} \right\} \, d\sigma_0(s) \\
    &\leq \sum_{q=1}^\infty n_q^{1-(\varepsilon_{j_q}/2)}\sigma_0(A_t(q)) \\
    &\leq \sum_{q=1}^\infty n_q^{1-(\varepsilon_{j_q}/2)}\sigma_0(B(t,b_q)) \\
    &\leq \frac{c_1}{m_f} \sum_{q=1}^\infty n_q^{1-(\varepsilon_{j_q}/2)}\left( \frac{\ln n_q}{n_q} \right) ^{1-(\varepsilon_{j_q}/4)} \\
    &\leq \frac{c_1}{m_f} \sum_{q=1}^\infty \frac{\ln n_q}{(n_q\ln n_q)^{\varepsilon_{j_q}/4}} <\infty 
\end{align*}
by Lemma \ref{lem:inhomocantor}. 
Hence, 
\[
\int_{K_f} \int_{K_f} \exp \left\{ m_f\sum_{k=1}^\infty (\ell_k - |t-s|)_+\right\} \, d\sigma_0(s)d\sigma_0(t)<\infty ,
\]
and we thus obtain that $\mathrm{Cap}_{m_f}(K_f)>0$. 

\subsubsection{Proof of Lemma \ref{lem:inhomocantor}}

We will define $A$ as a Cantor set by a modification of the standard construction using nested sequence. 
(See \cite{Falconer1990} for instance.) 
For $k\in \mathbb{N}$, let $\lambda _k=\sum _{i=1}^k1/3^i$. 
Let $E^{(0)}=[0,1]$. 
Define $E^{(1)}_{0}$ and $E^{(1)}_{2}$ by deleting from the middle of $E^{(2)}$ an open interval $J_0$ of length $\frac{1}{3} |E^{(0)}|$, namely $J_0=(\frac{1}{3},\frac{2}{3})$,  $E^{(1)}_{0}=[0,\frac{1}{3}]$ and $E^{(1)}_{2}=[\frac{2}{3},1]$, respectively. 
Continue this procedure up to $k_1$-steps. 
Namely, if we have defined closed intervals $E^{(k)}_{w}$ for $w=w_1\cdots w_{k}\in \{0,2\}^{k}$, where $k\in \{ 1,\dots ,k_1-1\}$, then define $E^{(k+1)}_{w0},E^{(k+1)}_{w2}$ by deleting from the middle of $E^{(k)}_{w}$ an open interval $J_{w1}$ of length $\frac{1}{3} |E^{(k)}_{w}|$, that is $E^{(k)}_{w}=E^{(k+1)}_{w0}\sqcup J_{w1}\sqcup E^{(k+1)}_{w2}$. 

In the next step, define closed intervals $E^{(k_1+1)}_{w0},E^{(k_1+1)}_{w2}$ for each $w\in \{0,2\}^{k_1}$ by deleting from the middle of $E^{(k_1)}_{w}$ an open interval $J_{w1}$ of length $(1-2\lambda _{k_1+1}) |E^{(k_1)}_{w}|$. 
Continue this procedure for the next $k_2$-steps. 
That is, if we have defined closed intervals $E^{(k)}_{w}$ for $w=w_1\cdots w_{k}\in \{0,2\}^{k}$, where $\in \{ k_1+1,\dots ,k_2-1\}$, then define $E^{(k+1)}_{w0},E^{(k+1)}_{w2}$ by deleting from the middle of $E^{(k)}_{w}$ an open interval $J_{w1}$ of length $(1-2\lambda _{k+1}) |E^{(k)}_{w}|$, that is $E^{(k)}_{w}=E^{(k+1)}_{w0}\sqcup J_{w1}\sqcup E^{(k+1)}_{w2}$. 

We will repeat the whole procedure above. 
Namely, for $k\in \{ k_2+1,\dots ,k_3\}$ we define $E^{k}_w$ by deleting an open interval of length $\frac{1}{3}$ portion of previous, and for $k\in \{ k_3+1,\dots ,k_4\}$ we define $E^{k}_w$ by deleting an open interval of length $(1-2\lambda _k)$ portion of previous, and so on. 
Define then 
\[
A=\cap _{k=1}^\infty \cup _{w\in \{0,2\}^{k}}E^{(k)}_w.
\]

For each integer $k\in \mathbb{N}$, denote the length of the intervals $E^{(k)}_{w}$, $w\in \{ 0,2\}^{k}$ by $\delta_k$. 
(Set $\delta_0=1$.)
Let $T=\sqcup _{q=0}^\infty \{ k_{2q}+1,\dots ,k_{2q+1}\} =\sqcup _{q=0}^\infty T_q$ and $H=\sqcup _{q=0}^\infty \{ k_{2q+1}+1,\dots ,k_{2(q+1)}\} =\sqcup _{q=0}^\infty H_q$ with $k_0=0$. 
Then for $k\in T\sqcup H=\mathbb{N}$, there is a unique $q=q(k)\in \mathbb{N}\cup \{0\}$ such that $k\in T_q$ or $k\in H_q$. 
One can see that
\begin{equation} \label{diamE^k:inhomog}
    \delta_k=
    \begin{cases}
    \left( \dfrac{1}{3} \right) ^{k-k_{2q}}\delta _{k_{2q}}, & k\in T_q \\
    \lambda _{k}\cdots \lambda _{k_{2q+1}+1} \delta _{k_{2q+1}}, & k\in H_q, \end{cases}
\end{equation}
where 
\[
\lambda _{k}\cdots \lambda _{k_{2q+1}+1}=\prod _{j=k_{2q+1}+1}^{k}\frac{1}{2} \left( 1-\frac{1}{3^j}\right) \approx \left( \frac{1}{2} \right)^{k-k_{2q+1}} \exp \left\{ -\frac{1}{2} \left( \frac{1}{3^{k_{2q+1}}}-\frac{1}{3^k}\right) \right\}
\]
up to a uniform constant, and thus $\delta _k/\delta _{k-1}\in [1/3,1/2)$. 
Note that $M(A,\delta_k)\leq 2^k$ in view of the obvious covering by $\{E^{(k)}_{w}\} _{w\in \{0,2\}^{k}}$. (Here and below, recall the definition of box-counting dimension from Section \ref{preliminary}) 
Note also that given an interval of length $\delta _k$ there exist at most two intervals $E^{(k-1)}_{w}$ which intersect it, and hence $M(A,\delta_k)\geq 2^{k-1}$. 
Thus for a sequence $k_q$ growing fast enough one can show that
\begin{align*}
    \underline{\dim} _\mathrm{B}A&=\liminf_{k\to \infty}\frac{\ln M(A,\delta_k)}{-\ln \delta _k}=\frac{\ln 2}{\ln 3} <1, \\
    \overline{\dim} _\mathrm{B}A&=\limsup_{k\to \infty}\frac{\ln M(A,\delta_k)}{-\ln \delta _k}=\frac{\ln 2}{\ln 2}=1. 
\end{align*}
But then $\dim _\mathrm{H}A\leq \underline{\dim} _\mathrm{B}A<\overline{\dim} _\mathrm{B}A=1$ .

Next, we define a distribution function $F$ on $[0,1]$ to obtain a measure supported on the Cantor set $A$ defined above, and it is accomplished in a similar way as defining the singular Cantor function. 
For completeness, it is outlined below. 
For $w\in \{0,2\}$, let
\[
\bar{w}=
\begin{cases}
0, & w=0, \\
1, & w=2.
\end{cases}
\]
Define for $x\in [0,1]\setminus A$ by
\[
F(x)=
\begin{cases}
\dfrac{1}{2}, & x\in J_1, \\
\sum_{i=1}^{k-1}\dfrac{\bar{w_i}}{2^i} +\dfrac{1}{2^k}, & x\in J_{w_1\cdots w_{k-1}1}.
\end{cases}
\]
It can be extended to a continuous non-decreasing function on $[0,1]$, and denote the resulting function by $F$ as well. 
Let $\sigma_0$ be the associated probability measure supported on $A$, i.e.
\[
\sigma_0((a,b))=F(b)-F(a).
\]
Then one can see that 
\[
\sigma_0(E^{(k)}_w)=\left( \frac{1}{2}\right)^k
\]
for every $w\in \{0,2\}^{k}$ and $k\in \mathbb{N}$. 
Taking $\# T_q<\# H_q$ and $\# H_q<\# H_{q+1}$ for all large $q\in \mathbb{N}$, it then follows from \eqref{diamE^k:inhomog} that 
\[
\sigma_0(E^{(k)}_w)\lesssim \delta _k^{1-(\varepsilon_k/4)}
\]
for $k\in H$ up to a uniform constant. 
Hence, there is a family of disjoint intervals $\{ [u_k,v_k]\}_{k\in \mathbb{N}}$ with $u_k>0$,  $v_{k+1}<u_k$, and $u_k/v_k<c_0$ for some constant $c_0>0$ such that for every $r\in [u_k,v_k]$ and $t\in A$ one has
\[
\sigma_0(B(t,r))\leq c_1r^{1-(\varepsilon_k/4)} 
\]
for some constant $c_1>0$. 
Lemma \ref{lem:inhomocantor} is obtained.

\section{Examples and corollaries} \label{examples}

We have the following corollary from Proposition \ref{mainProp:DtoS}.

\begin{proposition}\label{Cor:uD_nuD}
If for a sequence $\{\ell_n\}_{n\in \mathbb{N}}$ there is no Dvoretzky covering in the classical sense, i.e. when $f \equiv 1$, then there can be no Dvoretzky covering with respect to any absolute continuous measure. 
In particular, for the sequence $\ell_n=\left(\frac{1}{n}-\frac{\beta}{n \ln n}\right)$, with $\beta>1$, condition \eqref{mD} is fulfilled, but Theorem \ref{mainThm:perturbation} fails.
\end{proposition}

\begin{corollary}
In fact one can show that Theorem \ref{mainThm:perturbation} is not true if $K_f$ contains an interval.
\end{corollary}

As we saw in Theorem \ref{mainThm:h_d_fail1}, the condition $\dim_\mathrm{B}$ can not be replaced with $\dim_\mathrm{H}$. However, we now show that the theorem is true under more restrictive assumptions on $\{\ell_n\}_{n\in \mathbb{N}}$:

\begin{theorem} \label{mainThm:h_d_fail2} 
Assume $m_f>0$, $\dim_\mathrm{H} K_f<1$, and  \begin{equation} \label{cond:seq}
    \dim _\mathrm{H}K_f<\liminf _{n\to \infty}m_f n\ell_n\leq \limsup _{n\to \infty}m_f n\ell_n<\infty.
\end{equation}
Then the circle is $\mu_f$-Dvoretzky covered if and only if \eqref{mD} holds.
However, for the sequence $\ell_n =\frac{1}{m_f}\Big( \frac{1}{n}-\frac{1+\beta }{n \ln n}\Big)$ $(\beta >1)$ there exists a density function $f$, so that $m_f>0$, $\dim_\mathrm{H} K_f=1$, and \eqref{mD} is fulfilled but $\mathbb{T}$ is not $\mu_f$-Dvoretzky covered.
\end{theorem}

We will next prove the following corollary from Theorems \ref{mainThm:perturbation} and \ref{mainThm:h_d_fail1}:

\begin{corollary}\label{mainThm:necS-unnecK}
Condition \eqref{cond:shepp^a} is necessary for  $\mu_f$-Dvoretzky covering, even when $m_f=0$. However, \eqref{cond:kahane^m} is not always necessary for covering, i.e. there exists a sequence $\{\ell_n\}_{n \in \mathbb{N}}$ and a density function $f$ so that $\mathbb{T}$ is $\mu_f$ Dvoretzky covered, however \eqref{cond:kahane^m} is not full-filed.
\end{corollary}

\subsection{Proof of Proposition \ref{Cor:uD_nuD}}

Assume the opposite. 
There exists a density function $f$ for which there is $\mu_f$-Dvoretzky covering, but no Dvoretzky covering in the classical sense, i.e.
\[
\sum_{n=1}^\infty\frac{1}{n^2}e^{(\ell_1 + \dots + \ell_n)}<\infty.
\]
Since $\int_\mathbb{T}f\, dt=1$ and $f\not\equiv 1$, one has $m_f<1$. 
Then for every $a$, with $m_f<a\leq 1$ we have
\[
\sum_{n=1}^\infty\frac{1}{n^2}e^{a(\ell_1 + \dots + \ell_n)}<\sum_{n=1}^\infty\frac{1}{n^2}e^{(\ell_1 + \dots + \ell_n)}<\infty.
\]
This contradicts Proposition \ref{mainProp:DtoS}. 

Next we prove the second half. 
Note that the sequence $\ell_n=\left( \frac{1}{n}-\frac{{\beta}}{n \ln n}\right)$ satisfies condition \eqref{mD} since 
\[
\sum_{k=1}^n \left( \frac{1}{k}-\frac{{\beta}}{k \ln k}\right)\approx \ln n -\beta \ln \ln n
\] 
up to a uniform constant. 
We also have that
\[
\sum_{n=1}^\infty \frac{1}{n^2}e^{\sum_{k=1}^n \left( \frac{1}{k}-\frac{{\beta}}{k \ln k}\right)}\approx \sum_{n=1}^\infty \frac{1}{n^2}e^{\ln n -\beta \ln \ln n}=\sum_{n=1}^\infty \frac{1}{n \ln^\beta n}<\infty,
\]
whenever $\beta>1$. 
Thus, the sequence $\{\left(\frac{1}{n}-\frac{{\beta}}{n \ln n}\right)\}_{n \in \mathbb{N}}$ satisfies condition \eqref{mD}, and for it there is no Dvoretzky covering in the classical sense by the Shepp criterion. 
Hence, the uniform density cannot be modified to obtain non-uniform Dvoretzky covering similar to Theorem \ref{mainThm:perturbation}.

\subsection{Proof of Theorem \ref{mainThm:h_d_fail2}: Part I}

In this section, we will prove the first half of Theorem \ref{mainThm:h_d_fail2}. 
In Theorem \ref{mainThm:h_d_fail1} we saw that if the sequence $m_f(\ell_1 + \dots +\ell_n)/\ln n$ is strongly oscillatory and $\{\ell_n\}$ tends to have many clusters, then non-covering phenomena can occur when $\dim_\mathrm{H} K_f<1$. 
In contrast, in this section we show that if we restrict these oscillatory behaviour then $\mu_f$-Dvortezky covering can be achieved.

First, note that the necessity of \eqref{mD} follows from Theorem \ref{mainThm:ub_dvoretzky-hawkes}-(a) independently of the dimension of $K_f$ and \eqref{cond:seq}. 
It is thus enough to show that $\mathbb{T}$ is $\mu_f$-Dvoretzky covered by sequences satisfying   \eqref{mD} and \eqref{cond:seq}. 

To see $\mathbb{T}$ is $\mu_f$-Dvoretzky covered, we will use Proposition \ref{cond:fk}. 
We obtain \eqref{cond:shepp^a} by Lemma \ref{lem:shepp^a-hawkes}. 
Henceforth, we will show \eqref{cond:kahane^m}, that is,  $\mathrm{Cap}_{m_f}(K_f)=0$. 

For $(s,t)\in K_f\times K_f$ with $|t-s|\in (\ell_{n+1}, \ell_n]$ one has 
\[
m_f\sum_{k=1}^\infty (\ell_k - |t-s|)_+=m_f\sum_{k=1}^{n} (\ell_k - |t-s|). 
\]
For $t\in K_f$ and $n\in \mathbb{N}$, let $A_t(n)=(t+\ell_{n+1},t+\ell_{n}]\cup [t-\ell_{n},t-\ell_{n+1})$. 
Let $\nu_*=\liminf _{n\to \infty}m_f n\ell_n$ and $\nu^*=\limsup _{n\to \infty}m_f n\ell_n$. 
By assumption, for some constant $c_0\in (0,\nu^*)$ and $N_1\in \mathbb{N}$, we have $m_f n\ell_n\leq c_0$ for every $n\geq N_1$. 
Take $\varepsilon >0$ so small that $\nu_*-\varepsilon >\dim _\mathrm{H}K_f$. 
There exists $N_2=N_2(\varepsilon)\in \mathbb{N}$ such that for every $n\geq N_2$ one has $m_f n\ell_n >\nu_*-\varepsilon$. 
Let $N_0=\max \{ N_1,N_2\}$. 
Then we have 
\begin{align}
    I^{(m_f)}(\sigma)&=\int_{K_f} \int_{K_f} \exp \left\{ m_f\sum_{k=1}^\infty (\ell_k - |t-s|)_+\right\} \, d\sigma(s)d\sigma(t) \nonumber \\
    &=\sum_{n=1}^\infty \int_{K_f} \left[ \int_{K_f\cap A_t(n)} \exp \left\{ m_f\sum_{k=1}^{n} (\ell_k - |t-s|)\right\} \, d\sigma(s)\right] d\sigma (t) \nonumber \\
    &\geq \sum_{n=1}^{N_0-1} \int_{K_f} \left[ \int_{K_f\cap A_t(n)} \exp \left\{ m_f\sum_{k=1}^{n} (\ell_k - |t-s|)\right\} \, d\sigma(s)\right] d\sigma (t) \nonumber \\
    &\quad +\sum_{n=N_0}^{\infty } \int_{K_f} \left[ \int_{K_f\cap A_t(n)} \exp \left\{ \left( m_f\sum_{k=1}^{n} \ell_k\right) - c_0\right\} \, d\sigma(s)\right] d\sigma (t), \label{principal_kernel}
\end{align}
where the first term in the right-hand side, denoted by $R_0$, will be bounded from above for 
\[
0\leq R_0\leq \sum_{n=1}^{N_0-1} \int_{K_f} \left[ \int_{K_f\cap A_t(n)} e^{m_f n\ell _1} \, d\sigma(s)\right] d\sigma (t) \leq N_0|K_f|^2e^{m_fN_0}<\infty .
\]
For the second term \eqref{principal_kernel}, it follows that 
\begin{align*}
    m_f\sum_{k=1}^{n} \ell_k
    =m_f\sum_{k=1}^{N_2-1} \ell_k+m_f\sum_{k=N_2}^{n} \ell_k
    &\geq m_f\sum_{k=1}^{N_2-1} \ell_k+\sum_{k=N_2}^{n} \frac{\nu_*-\varepsilon}{k} \\
    &=m_f\sum_{k=1}^{N_2-1} \ell_k - \sum_{k=1}^{N_2-1} \frac{\nu_*-\varepsilon}{k} + \sum_{k=1}^{N_2-1} \frac{\nu_*-\varepsilon}{k} + \sum_{k=N_2}^{n} \frac{\nu_*-\varepsilon}{k} \\
    &=\sum_{k=1}^{N_2-1} \left( m_f\ell_k -  \frac{\nu_*-\varepsilon}{k} \right) + \sum_{k=1}^{n} \frac{\nu_*-\varepsilon}{k} ,
\end{align*}
and hence 
\begin{equation} \label{ineq:kernel1}
    \exp \left\{ \left( m_f\sum_{k=1}^{n} \ell_k\right) - c_0\right\} \geq C_1 \exp \left\{ \sum_{k=1}^{n} \frac{\nu_*-\varepsilon}{k} \right\} \geq C_1n^{\nu_*-\varepsilon},
\end{equation}
where $C_1=C_1(\varepsilon)>0$ is a bounded constant taken as 
\[
C_1=e^{-c_0} \exp\left\{ \sum_{k=1}^{N_2-1} \left( m_f\ell_k -  \frac{\nu_*-\varepsilon}{k} \right)\right\} \leq \exp\left\{ \sum_{k=1}^{N_2-1} m_f\ell_k\right\} \leq e^{m_fN_1}<\infty . 
\]
Note that for every $s\in A_t(n)$ and $n\geq N_0$, one has 
\begin{equation} \label{ineq:kernel2}
    n|t-s|\geq c_1>0
\end{equation}
for some constant $c_1=c_1(\varepsilon)>0$ as 
\[ 
    n|t-s|\geq n\ell_{n+1}=\frac{n}{n+1} (n+1)\ell_{n+1}\geq \frac{n}{n+1}  \frac{\nu_*-\varepsilon}{m_f} >0. 
\]

It follows from \eqref{principal_kernel}, \eqref{ineq:kernel1}, and \eqref{ineq:kernel2} that 
\begin{align*}
    I^{(m_f)}(\sigma)
    &\geq R_0+C_2\sum_{n=N_0}^{\infty } \int_{K_f} \left[ \int_{K_f\cap A_t(n)}  \frac{1}{|t-s|^{\nu_*-\varepsilon}} \, d\sigma(s)\right] d\sigma (t) \\
    &=R_0+C_2\left\{ \int_{K_f}\int_{K_f} \frac{1}{|t-s|^{\nu_*-\varepsilon}}\,d\sigma(t)d\sigma(s)-\iint_{B}\frac{1}{|t-s|^{\nu_*-\varepsilon}}\,d\sigma(t)d\sigma(s)\right\} 
\end{align*}
where we let $C_2=C_1 c_1^{\nu_*-\varepsilon}\in (0,\infty)$, and $B=\cup_{t\in K_f}\cup _{n=1}^{N_0-1}\{ \{t\}\times ( K_f\cap A_t(n))\}\subset K_f\times K_f$. 
Note that the double integral over $B$, denote it by $R_1$, is bounded since
\[
R_1=\iint_{B}\frac{1}{|t-s|^{\nu_*-\varepsilon}}\,d\sigma(t)d\sigma(s)\leq \sigma \otimes \sigma (B)\frac{1}{\ell _{N_0}{}^{\nu_*-\varepsilon}} <\infty .
\]
Thus we have
\begin{equation} \label{ineq:K-Renergies}
    I^{(m_f)}(\sigma)\geq R_0-C_2R_1+C_2\int_{K_f}\int_{K_f} \frac{1}{|t-s|^{\nu_*-\varepsilon}}\,d\sigma(t)d\sigma(s). 
\end{equation}
The last integral in \eqref{ineq:K-Renergies} is known as the Riesz energy of the measure $\sigma$ supported on a compact set $F$:
\begin{equation}\label{Riesz}
    I_s(\sigma)=\int_F\int_F \frac{1}{|t-s|^{s}} \, d\sigma(t)d\sigma(s).
\end{equation}
It is known that if $I_s(\nu)<\infty$ for some $\nu$, then $\dim_\mathrm{H} F\geq s$. (See \cite{Falconer1990}*{Proposition 4.13} for instance.) 
Hence, since $\nu_*-\varepsilon>\dim_\mathrm{H}K_f$, for every measure $\sigma$ supported on $K_f$ we will have $I_{\nu_*-\varepsilon}(\sigma)=\infty$, and thus $I^{(m_f)}(\sigma)=\infty$ by \eqref{ineq:K-Renergies}. 
This implies that $\mathrm{Cap}_{m_f}(K_f)=0$. 

\subsection{Proof of Theorem \ref{mainThm:h_d_fail2}: Part II} 

In this section, we will prove the second half of Theorem \ref{mainThm:h_d_fail2}. 
We show that there exist a compact set $A\subset \mathbb{T}$ with $\dim_\mathrm{H} A=1$, and a probability density function $f\in L^1(\mathbb{T})$ with $m_f>0$ and $f$ is Lipschitz on $A$ such that $K_f=A$. 
Then by \cite{Fan-Karagulyan2021}*{Corollary 1.1-(2)}, we see that the conditions \eqref{cond:shepp^a} and \eqref{cond:kahane^m} will be both necessary and sufficient for $\mu_f$-Dvoretzky covering. 
However, since we can show $\mathrm{Cap}_{m_f}(K_f)>0$ for the sequence $\ell_n=\frac{1}{m_f}\left(\frac{1}{n}-\frac{1+\beta}{ n\ln n}\right)$, there can be no $\mu_f$-Dvoretzky covering. 
It now remains to construct the set $A$.

\begin{lemma} \label{sing-Frostman}
    There exists a Cantor set $A\subset \mathbb{T}$, with $\dim_\mathrm{H}A=1$ and $|A|=0$, admitting a Borel probability measure $\sigma_0$ supported on $A$ for which 
    \begin{equation} \label{ineq:sing-Frostman}
        \sigma_0(B)\leq c_0|\ln\diam B|\cdot \diam B
    \end{equation}
    for every ball $B\subset \mathbb{T}$ and some constant $c_0>0$. 
\end{lemma}
We will prove this lemma at the end of this section, and proceed the argument to show $A(=K_f)$ has positive capacity with respect to $\Phi^{(m_f)}$, namely 
\begin{equation} \label{CapK>0}
    I^{(m_f)}(\sigma_0)=\int_{A} \int_{A} \exp \left\{ m_f\sum_{k=1}^\infty (\ell_k - |t-s|)_+\right\} \, d\sigma_0(s)d\sigma_0(t)<\infty.
\end{equation}
Note that for $(s,t)\in A\times A$ with $|t-s|\in (\ell_{n+1}, \ell_n]$ one has 
\begin{equation*}
    m_f\sum_{k=1}^\infty (\ell_k - |t-s|)_+
    =m_f\sum_{k=1}^{n} (\ell_k - |t-s|)\leq m_f\sum_{k=1}^{n} \ell_k\approx \ln n-(1+\beta )\ln \ln n
\end{equation*}
up to a uniform constant. 
For $t\in A$ and $n\in \mathbb{N}$, let $A_t(n)=(t+\ell_{n+1},t+\ell_{n}]\cup [t-\ell_{n},t-\ell_{n+1})$. 
It follows that for each $t\in A$ 
\begin{align*}
    \int_{A} \exp \left\{ m_f \sum_{k=1}^\infty (\ell_k - |t-s|)_+\right\} \, d\sigma_0(s)     &=\sum_{n=1}^\infty\int_{A\cap A_t(n)} \exp \left\{ m_f \sum_{k=1}^\infty (\ell_k - |t-s|)_+\right\} \, d\sigma_0(s) \\
    &\lesssim \sum_{n=1}^\infty\int_{A\cap A_t(n)} \frac{n}{(\ln n)^{1+\beta}} \, d\sigma_0(s) \\
    &\leq \sum_{n=1}^\infty \left( \frac{n}{(\ln n)^{1+\beta}} \right)  \sigma_0(A_t(n)).
\end{align*}
Since $\ell_n-\ell_{n+1}\lesssim 1/n^2$, it follows from \eqref{ineq:sing-Frostman} that 
\[
\sigma_0(A_t(n))\lesssim 4c_0(\ln n)\cdot \frac{1}{n^2}, 
\]
and thus
\[
\int_{A} \exp \left\{ m_f \sum_{n=1}^\infty (\ell_k - |t-s|)_+\right\} \, d\sigma_0(s)\lesssim \sum_{n=1}^\infty \frac{1}{n(\ln n)^{\beta}} 
\]
for every $t\in A$. 
Consequently, one has
\begin{align*}
    I^{(m_f)}(\sigma_0)\lesssim \sigma_0(A)\sum_{n=1}^\infty \frac{1}{n(\ln n)^{\beta}} <\infty
\end{align*}
for $\beta >1$. 

\subsubsection{Proof of Lemma \ref{sing-Frostman}} 

First, define a Cantor set having Hausdorff dimension one and of Lebesgue measure zero by the standard construction using nested sequence. 
(See \cite{Falconer1990} for instance.) 
Let $E^{(1)}=E^{(2)}=[0,1]$. 
Define $E^{(3)}_{0}$ and $E^{(3)}_{2}$ by deleting from the middle of $E^{(2)}$ an open interval $J_1$ of length $\frac{1}{3} |E^{(2)}|$, namely $J_1=(\frac{1}{3},\frac{2}{3})$,  $E^{(3)}_{0}=[0,\frac{1}{3}]$ and $E^{(3)}_{2}=[\frac{2}{3},1]$, respectively. 
In the next step, define closed intervals $E^{(4)}_{00},E^{(4)}_{02}$ by deleting from the middle of $E^{(3)}_{0}$ an open interval $J_{01}$ of length $\frac{1}{4} |E^{(3)}_{0}|$. 
By the same way, define closed intervals $E^{(4)}_{20},E^{(4)}_{22}$ by deleting from the middle of $E^{(3)}_{2}$ an open interval $J_{21}$ of length $\frac{1}{4} |E^{(3)}_{2}|$. 
If we have defined closed intervals $E^{(k+1)}_{w}$ for $w=w_1\cdots w_{k-1}\in \{0,2\}^{k-1}$, then define $E^{(k+2)}_{w0},E^{(k+2)}_{w2}$ by deleting from the middle of $E^{(k+1)}_{w}$ an open interval $J_{w1}$ of length $\frac{1}{k} |E^{(k+1)}_{w}|$, that is $E^{(k+1)}_{w}=E^{(k+2)}_{w0}\sqcup J_{w1}\sqcup E^{(k+2)}_{w2}$. 
Define 
\[
A=\cap _{k=1}^\infty \cup _{w\in \{0,2\}^{k}}E^{(k+2)}_w.
\]

For each integer $k\geq 3$, all the intervals $E^{(k)}_{w}$, $w\in \{ 0,2\}^{k-2}$, have the same length, and denote it by $\delta_k$. 
(Set $\delta_1=\delta_2=1$.)
In fact, for $k\geq 3$ one has
\[
\delta_k=\frac{1}{2} \left(1-\frac{1}{k}\right)\delta_{k-1},
\]
and thus $\delta_{k}/\delta_{k-1}\in [1/3,1/2)$. 
It follows that for all large $k\in \mathbb{N}$ 
\begin{equation} \label{diamE^k}
    \delta_k=\prod_{j=3}^k\frac{1}{2} \left(1-\frac{1}{j}\right) \approx \left(\frac{1}{2}\right)^{k-2}e^{-\sum_{j=3}^k\frac{1}{j}} \approx \left(\frac{1}{2}\right)^{k-2}\frac{1}{k}
\end{equation}
up to uniform constants. 
Hence one has $|A|=0$. 
By \cite{Falconer1990}*{Example 4.6}, one also has
\[
\dim_{\mathrm{H}}A\geq \liminf_{k\to \infty} \frac{\ln (2^k)}{-\ln \left( \frac{2}{k+1}\delta_k\right)} =1. 
\]

Next, we define a distribution function $F$ on $[0,1]$ to obtain a measure supported on the Cantor set $A$ defined above, and it is accomplished by the same way as in the proof of Lemma \ref{lem:inhomocantor}, hence we omit this part. 
Let $\sigma_0$ be the associated probability measure supported on $A$, i.e.
\[
\sigma_0((a,b))=F(b)-F(a).
\]
Then one can see that 
\[
\sigma_0(E^{(k+2)}_w)=\left(\frac{1}{2}\right)^{k}
\]
for every $w\in \{0,2\}^{k}$ and $k\in \mathbb{N}$. 
Since $\mathrm{diam}E^{(k+2)}_w=\delta_{k+2}$, it follows from \eqref{diamE^k} that 
\[
\sigma_0(E^{(k+2)}_w)\lesssim (k+2)\mathrm{diam}E^{(k+2)}_w\lesssim -\ln (\mathrm{diam}E^{(k+2)}_w)\cdot \mathrm{diam}E^{(k+2)}_w
\]
up to a uniform constant as $-\ln \delta_{k+2}=k\ln 2+\ln (k+2)$. 
Given a ball $B$ of diameter (length) $r>0$, it is covered by at most two $E^{(k+2)}_{w}$ and $E^{(k+2)}_{w'}$ for some $w,w'\in \{0,2\}^{k}$ where $k=k_r\in \mathbb{N}$ is chosen so that $\delta_{k+3}\leq r<\delta _{k+2}$. 
Thus it follows from above that 
\[
\sigma_0(B)\leq c_0|\ln (\mathrm{diam}B)|\cdot \mathrm{diam}B
\]
for some constant $c_0>0$. 
Lemma \ref{sing-Frostman} is proven.

\subsection{Proof of Corollary \ref{mainThm:necS-unnecK}}

Condition \eqref{cond:shepp^a} is necessary due to Proposition \ref{mainProp:DtoS}. 
To see that \eqref{cond:kahane^m} is not necessary, consider the density function $h$ constructed in Theorem \ref{mainThm:h_d_fail1}. 
Consider also the sequence $\{\ell_n\}_{n \in \mathbb{N}}$ constructed in Theorem \ref{mainThm:h_d_fail1}. 
We then have $\mathrm{Cap}_{m_h}(K_h)>0$. 
Note that $|K_h|=0$. 
Hence, by applying Theorem \ref{mainThm:perturbation} we can find a density function $h_0$ so that $m_{h_0}=m_h$, $K_{h_0}=K_h$ and for which there is $\mu_{h_0}$-Dvoretzky covering. 
Thus, for $h_0$ and $\{\ell_n\}_{n \in \mathbb{N}}$ above, there will be $\mu_{h_0}$-Dvoretzky covering despite the fact that $\mathrm{Cap}_{m_{h_0}}(K_{h_0})=\mathrm{Cap}_{m_h}(K_h)>0$. 
Taking $f=h_0$, we obtain Corollary \ref{mainThm:necS-unnecK}.

\end{document}